\documentclass[12pt]{amsart}
\pdfoutput=1
\usepackage{microtype}
\usepackage[active]{srcltx}
\usepackage{calc,amssymb,amsthm,amsmath,amscd, eucal,ulem}
\usepackage{alltt}
\usepackage{mathtools}
\usepackage{float}

\usepackage{bbding}
\RequirePackage[dvipsnames,usenames]{color}

\usepackage{tikz}

\usepackage{amsfonts, mathrsfs}
\usepackage[cal=boondox, calscaled=1.05]{mathalfa}

\usepackage{calligra}
\usepackage{amssymb}
\usepackage{stmaryrd}
\usepackage{dcpic,pictexwd}

\usepackage[left=1in,top=1in,right=1in,bottom=1in]{geometry}
\usepackage{bm}
\usepackage{verbatim}
\usepackage{upgreek}
\usepackage{systeme}
\usepackage{setspace}
\usepackage{hyperref}
\usepackage{enumerate}
\usepackage{enumitem}
\usepackage{subfig}
\usepackage{graphicx}
\usepackage[all,cmtip]{xy}
\usepackage{verbatim}

\RequirePackage{aliascnt}
\newcommand{\basetheorem}[3]{%
    \newtheorem{#1}{#2}[#3]
    \newtheorem*{#1*}{#2}
    \expandafter\def\csname #1autorefname\endcsname{#2}
}%
\newcommand{\maketheorem}[3]{%
    \newaliascnt{#1}{#3}
    \newtheorem{#1}[#1]{#2}
    \aliascntresetthe{#1}
    \expandafter\def\csname #1autorefname\endcsname{#2}
    \newtheorem{#1*}{#2}
}%
\theoremstyle{plain}   

\basetheorem{theorem}{Theorem}{section}

\maketheorem{proposition}{Proposition}{theorem}
\maketheorem{corollary}{Corollary}{theorem}
\maketheorem{lemma}{Lemma}{theorem}
\newtheorem*{mainthm*}{Main Theorem}
\newtheorem*{theoremA*}{Theorem A}
\newtheorem*{theoremB*}{Theorem B}
\maketheorem{claim}{Claim}{theorem}

\theoremstyle{definition}    

\maketheorem{definition}{Definition}{theorem}
\maketheorem{notation}{Notation}{theorem}
\maketheorem{example}{Example}{theorem}

\theoremstyle{remark}    

\maketheorem{remark}{Remark}{theorem}
\maketheorem{question}{Question}{theorem}
\maketheorem{convention}{Convention}{theorem}
\maketheorem{setup}{Setup}{theorem}

\numberwithin{equation}{theorem}

\DeclareMathOperator{\Gal}{{Gal}}
\DeclareMathOperator{\Exc}{{Exc}}
\DeclareMathOperator{\Bl}{{Bl}}
\DeclareMathOperator{\Cl}{{Cl}}
\DeclareMathOperator{\chern}{ch}
\DeclareMathOperator{\td}{td}
\DeclareMathOperator{\length}{{length}}
\DeclareMathOperator{\Spf}{{Spf}}
\DeclareMathOperator{\Stab}{{Stab}}
\DeclareMathOperator{\st}{\mathrm{st}}
\DeclareMathOperator{\utd}{\mathrm{utd}}
\DeclareMathOperator{\Spec}{{Spec}}
\DeclareMathOperator{\Aut}{{Aut}}
\DeclareMathOperator{\ord}{{ord}}
\DeclareMathOperator{\Image}{Image}
\DeclareMathOperator{\Char}{{char}}
\DeclareMathOperator{\Hom}{Hom}

\newcommand{\kay}{\mathcal{k}}

\newcommand{\sA}{\mathcal{A}}

\newcommand{\sE}{\mathcal{E}}

\newcommand{\sH}{\mathcal{H}}
\newcommand{\sI}{\mathcal{I}}
\newcommand{\sJ}{\mathcal{J}}
\newcommand{\sK}{\mathcal{K}}

\newcommand{\sM}{\mathcal{M}}
\newcommand{\sN}{\mathcal{N}}
\newcommand{\sO}{\mathcal{O}}

\newcommand{\sT}{\mathcal{T}}

\newcommand{\sX}{\mathcal{X}}
\newcommand{\sY}{\mathcal{Y}}
\newcommand{\sZ}{\mathcal{Z}}

\newcommand{\fram}{\mathfrak{m}}

\newcommand{\bA}{\mathbb{A}}

\newcommand{\bD}{\mathbb{D}}
\newcommand{\bE}{\mathbb{E}}

\newcommand{\bL}{\mathbb{L}}

\newcommand{\bN}{\mathbb{N}}

\newcommand{\bP}{\mathbb{P}}
\newcommand{\bQ}{\mathbb{Q}}

\newcommand{\bZ}{\mathbb{Z}}

\newcommand{\cf}{{\itshape cf.} }

\begin{document}
\title[Stringy motives under Galois quasi-\'etale covers]{On the behavior of stringy motives under Galois quasi-\'etale covers} 
\author[J.~Carvajal-Rojas]{Javier Carvajal-Rojas}
\address{Centro de Investigaci\'on en Matem\'aticas, A.C., Callej\'on Jalisco s/n, 36024 Col. Valenciana, Guanajuato, Gto, M\'exico}
\email{\href{mailto:javier.carvajal@cimat.mx}{javier.carvajal@cimat.mx}}

\author[T.~Yasuda]{Takehiko Yasuda}
\address{Department of Mathematics, Graduate School of Science, Osaka University, Toyonaka, Osaka, 560--0043, Japan} 
\email{\href{mailto:}{yasuda.takehiko.sci@osaka-u.ac.jp}}

\keywords{Kawamata log-terminal, stringy motive, \'etale fundamental group.}
\thanks{The first named author was partially supported by the grants ERC-STG \#804334, FWO Grant \#G079218N, CONAHCYT \#CBF2023-2024-224, and CONAHCYT \#CF-2023-G-33. The second named author was partially supported by JSPS KAKENHI Grant numbers 18K18710, 18H01112, and 16H06337.}
\subjclass[2020]{14E18, 14E30, 14E20, 14B05}

\begin{abstract}
We investigate the behavior of stringy motives under Galois quasi-\'etale covers. We prove that they descend under such covers in a sense defined via their Poincar\'e realizations. Further, we show that such descent is strict in the presence of ramification. As a corollary, we reduce the problem regarding the finiteness of the \'etale fundamental group of KLT singularities to a DCC property for their stringy motives. We verify such DCC property for surfaces in arbitrary characteristic. As an application, we give a characteristic-free proof for the finiteness of the \'etale fundamental group of log terminal surface singularities, which was unknown in equal characteristics $2$ and $3$ and in mixed characteristics. 
\end{abstract}
\maketitle

\section{Introduction} \label{sec.Introduction}

Kawamata log terminal singularities (KLT for short) are arguably the most important class of singularities in algebraic geometry. For instance, they are the gold standard for mild log canonical singularities and are the singularities required to run the Minimal Model Program. Hence, there has been a great effort in the last decades to understand how mild these singularities actually are. In general, much is known over fields of characteristic zero, but the situation over positive characteristic fields---let alone mixed characteristics---is rather thorny. A typical example of this is their rationality. We know that KLT singularities are rational in characteristic zero thanks to vanishing theorems such as Kodaira vanishings and their generalizations. However, they are not rational in general over positive characteristic fields and it is yet to be determined when exactly they are rational. Another but related problem has to do with the purity of the branch locus over KLT singularities, or, to be more precise, with the finiteness of their local \'etale fundamental groups. In this note, we study this problem by means of stringy motives; which is an invariant of KLT singularities amalgamating their log discrepancies with (generalized) Euler characteristics. Next, we recall what the problem inspiring this work is. 

Let $(X,\Delta)$ be a log pair over an algebraically closed field $\kay$ of characteristic $p\geq 0$ and $K_X$ be a canonical divisor on $X$. That is, $X$ is a normal $\kay$-variety and $\Delta$ is a $\bQ$-divisor on $X$ with coefficients in $[0,1]$ such that $K_{(X,\Delta)} \coloneqq K_X + \Delta$ is a $\bQ$-Cartier $\bQ$-divisor. We may also refer to $\Delta$ as a boundary on $X$. Let $r_{(X,\Delta)}$ denote the index of $(X,\Delta)$, i.e., the minimal positive integer $r$ such that $rK_{(X,\Delta)}$ is Cartier. A morphism $g\colon (X',\Delta') \to (X,\Delta)$ between log pairs is said to be a \emph{Galois quasi-\'etale log-cover} if $g$ is a finite dominant morphism such that $K_{X'}= g^*K_X$, $\Delta' = g^*\Delta$, and the extension of function fields $\sK(X')/\sK(X)$ is Galois. In such a case, $g$ is crepant and $r_{(X',\Delta')}$ divides $r_{(X,\Delta)}$. 

This work is concerned with the following fundamental problem.

\begin{question} \label{que.questionIntro.}
Assume that $(X,\Delta)$ is a Kawamata log terminal log pair and consider a tower of Galois quasi-\'etale log-covers
\begin{equation}\label{eqn.tower}
(X, \Delta) = (X_0, \Delta_0) \xleftarrow{f_0} (X_1, \Delta_1) \xleftarrow{f_1} (X_2, \Delta_2) \xleftarrow{f_2} \cdots
\end{equation} 
Does there exist $N \in \bN$ such that $f_i$ is \'etale for all $i\geq N$?
\end{question}
\begin{remark}
In \autoref{que.questionIntro.}, the compositions $g_n \coloneqq f_n \circ \cdots \circ f_0$ are the ones being required to be Galois quasi-\'etale log-covers (not the $f_i$'s), which then implies that the maps $f_i$ are Galois quasi-\'etale. Further, the Galois hypothesis is essential; see \cite[Proposition 11.4]{GrebKebekusPeternellEtaleFundamental}. Additionally, if $(X,\Delta)$ is KLT then so is $(X_i,\Delta_i)$ for all $i$; see \cite[Corollary 2.43]{KollarSingulaitieofMMP}.
\end{remark}

\autoref{que.questionIntro.} has attracted great attention in recent years; see \cite{GrebKebekusPeternellEtaleFundamental, XuFinitenessOfFundGroups, BhattCarvajalRojasGrafSchwedeTucker}. It is well-known that \autoref{que.questionIntro.} is intimately related to the problem of determining the finiteness of the (regional) local \'etale fundamental group of KLT singularities; see \cite{StibitzFundamentalGroups}. This ``twin'' problem has been considered in the works \cite{XuFinitenessOfFundGroups, CarvajalSchwedeTuckerEtaleFundFsignature,CarvajalRojasStablerKollarFundamentalGroupKLTthreefolds,BraunFundamentalGroupsKLT,BraunFilipazziMoragaSvaldiJordanProperty,CaiLeeTuckerSchwedePerfectoidSignature}, \cf \cite[Corollary 1.9]{KawamataTheConeOfCurves} which settles this for surfaces in characteristic zero. Despite these efforts, \autoref{que.questionIntro.} remains open in positive characteristics. Nonetheless, we know its answer to be affirmative for $F$-regular pairs (in full generality) as well as in dimensions $\leq 3$ but characteristics $\geq 7$; see \cite{BhattCarvajalRojasGrafSchwedeTucker,CarvajalRojasStablerKollarFundamentalGroupKLTthreefolds}. In dimension $2$, we may add $p=5$ to the known cases by putting together \cite{KawamataIndex1CoversOfLogTerminalSurfaceSingularities,ArtinCoveringsOfTheRtionalDoublePointsInCharacteristicp}; see \autoref{rem.FundamentalGropuSurfaceCase}. 

It is worth noting that, in dimensions $\leq 3$, \autoref{que.questionIntro.} can be answered in the tame case due to recent advances in the minimal model program; see \cite[Proposition 5.2]{HaconWitaszekMMPLowCharacteristic}, \cf \cite[Theorem 3.4]{XuZhangNonvanishing}.
Note also that the local \'etale fundamental group of an $F$-regular singularity as well as the one of a KLT singularity in dimension $3$ and characteristic $\ge 7$ is not only finite but tame as shown in \cite{CarvajalSchwedeTuckerEtaleFundFsignature,CarvajalRojasStablerKollarFundamentalGroupKLTthreefolds}. On the other hand, even in dimension two, there are rational double points (hence KLT singularities) having no non-trivial tame quasi-\'etale cover, equivalently, whose local \'etale fundamental group has no non-trivial tame quotient \cite[p.15]{ArtinCoveringsOfTheRtionalDoublePointsInCharacteristicp}. In general, the wild aspects in low characteristics make these problems much more difficult. For instance, Kawamata's method using index-$1$ covers does not apply to characteristics $\le 3$ in dimension $2$ as log terminal singularities are not preserved by index-$1$ covers. Thus, a new approach is needed in this case which we provide here. 

In this work, we propose a novel strategy to attack \autoref{que.questionIntro.} via \emph{stringy motives}; at least in dimensions $\leq 3$ where resolution of singularities and ($W\sO$-)rationality of KLT singularities are available. Loosely speaking, stringy motives are a hybrid between log discrepancies and Euler characteristics and are well-suited to the study of wild quotients.

We sketch our approach next. Following \cite[\S 6]{YasudaPcyclicMackayCorrespondenceMotivicIntegration}, we consider the stringy motive 
\[
M_{\mathrm{st}}(X,\Delta) \coloneqq \int_{J_{\infty}X}\bL^{F_{(X,\Delta)}} \mathrm{d}\mu_X \in \hat{\sM}'_{\kay,r_{(X,\Delta)}}
\]
of a log pair $(X,\Delta)$ which, by definition, is convergent (i.e., exists) if and only if $(X,\Delta)$ is \emph{stringily KLT}. Such pairs form a subclass of KLT log pairs. Conversely, a KLT log pair is stringily KLT if it has a log resolution. Hence, these two notions agree in characteristic zero and in dimensions $\leq 3$ in positive characteristics. 

Then, our goal is to examine the behavior of stringy motives under Galois quasi-\'etale log-covers $g\colon (X',\Delta') \to (X,\Delta)$. Let $G \coloneqq \Gal\big(\sK(X')/\sK(X)\big)$ be the corresponding Galois group. To do so, we are going to require a notion of ordering among stringy motives, for which we use the \emph{Poincar\'e realization}, which is a ring homomorphism:
\[
P \colon \hat{\sM}'_{\kay,r} \to \bZ\llparenthesis T^{-1/r} \rrparenthesis
\]
where $r \coloneqq r_{(X,\Delta_X)}$ and $\bZ\llparenthesis T^{-1/r} \rrparenthesis$ is the ring of integral Laurent series on $T^{-1/r}$. Thus, to a stringy motive we may associate its Poincar\'e realization $P_{\mathrm{st}}(X,\Delta_X) \coloneqq P\big(M_{\mathrm{st}}(X,\Delta_X);T \big) \in \bZ\llparenthesis T^{-1/r} \rrparenthesis$. For instance, in the surface case, any stringy motive we consider here will belong to the subring $\bZ \llparenthesis \bL^{-1/r} \rrparenthesis \subset \hat{\sM}'_{\kay,r}$ on which the Poincaré realization map is given by $\bL \mapsto T^2$ (where $\bL \coloneqq \{\bA^1\}$)

We obtain an ordering on $\bZ\llparenthesis T^{-1/r} \rrparenthesis$ as follows. For $0 \neq f \in \bZ\llparenthesis T^{-1/r} \rrparenthesis$, $f>0$ if and only if for $f = \sum_{i=-n}^{\infty} a_i T^{-i/r}$ we have $a_{-n}>0$. This is none other than the lexicographic ordering. We lift this ordering to the ring $\hat{\sM}'_{\kay,r}$ via the Poincar\'e realization. For instance, $M_{\mathrm{st}}(X,\Delta_X) > 0$ means $P_{\mathrm{st}}(X,\Delta_X)>0$. In general, we shall say that an element of $\hat{\sM}'_{\kay,r}$ (e.g. motivic measures and integrals) is positive if so is its Poincaré realization.

Since $(Y,\Delta_Y)$ admits a log $G$-action (with log quotient $(X,\Delta_X)$) so does $M_{\mathrm{st}}(Y,\Delta_Y)$. Then, we may form the quotient $M_{\mathrm{st}}(Y,\Delta_Y)/G$; see \autoref{sec.StringyMotives}. Our first observation is the equality:
\[
M_{\mathrm{st}}(Y,\Delta_Y)/G = \int_{g_{\infty}(J_{\infty} Y)}{\bL^{F_{X,\Delta_X}}}\mathrm{d}\mu_X ,
\]
where $g_{\infty} \colon J_{\infty} Y \to J_{\infty} X$ is the induced morphism on the spaces of arcs; see \autoref{cor.KeyCorollary}. Consequently, $M_{\mathrm{st}}(X,\Delta_X)$ and $M_{\mathrm{st}}(Y,\Delta_Y)/G $ can be compared as follows:
\begin{align*}
    M_{\mathrm{st}}(X,\Delta_X) = \int_{J_{\infty} X}{\bL^{F_{X,\Delta_X}}}\mathrm{d}\mu_X  &= \int_{g_{\infty}(J_{\infty} Y)}{\bL^{F_{X,\Delta_X}}}\mathrm{d}\mu_X  +\int_{J_{\infty} X \setminus g_{\infty}(J_{\infty} Y)}{\bL^{F_{X,\Delta_X}}}\mathrm{d}\mu_X \\
    &= M_{\mathrm{st}}(Y,\Delta_Y)/G + \int_{J_{\infty} X \setminus g_{\infty}(J_{\infty} Y)}{\bL^{F_{X,\Delta_X}}}\mathrm{d}\mu_X .
\end{align*}
Our main result is then the following.
\begin{mainthm*}[\autoref{thm.StringyInvariantGoesDown}]
With notation as above, assume that $(X,\Delta_X)$ is a KLT pair of dimension $\leq 3$ and that $g$ is not \'etale. Further, assume $G$ to be a $p$-group if $d=3$ and $0<p \leq 5$. Then, the motivic integral $\int_{J_{\infty} X \setminus g_{\infty}(J_{\infty} Y)}{\bL^{F_{X,\Delta_X}}}\mathrm{d}\mu_X $ is positive and so:
\[
 M_{\mathrm{st}}(X,\Delta_X) > M_{\mathrm{st}}(Y,\Delta_Y)/G.
\]
\end{mainthm*}

In particular, if there were a non-affirmative answer to \autoref{que.questionIntro.} in any of the following cases:
\begin{enumerate}
    \item $\dim X = 2$ and $p \geq0$,
    \item $\dim X = 3$ and $p \notin \{2,3,5\}$, 
    \item $\dim X =3$, $p \in \{2,3,5\}$, and $\Gal(X_i/X)$ is a $p$-group for all $i \geq 0$,
\end{enumerate}
then, the corresponding tower \autoref{eqn.tower} would yield a strictly descending chain:
\[
M_{\mathrm{st}}(X_0, \Delta_0)/G_0 > M_{\mathrm{st}}(X_1, \Delta_1)/G_1 > M_{\mathrm{st}}(X_2, \Delta_2)/G_2 > \cdots
\]
where $G_i \coloneqq \Gal\big(K(X_i)/(X_0)\big)$. Hence, we reach a contradiction if we can prove a descending chain condition (DCC) for stringy motives in those cases. We are able to achieve this in case (a); see \autoref{prop:DCC}, but leave it open in the remaining two cases. Noteworthy, a DCC property for another version of stringy invariants was discussed by Takahashi; see \cite{TakahashiDCCStringyInvariants}. 

As an application, we give a characteristic-free proof for the finiteness of the local \'etale fundamental group of KLT surface singularities; see \autoref{thm.ApplicationFundamentalGroups}. Of course, if the (purely wild) DCC condition holds for threefolds, the same would hold for KLT threefold singularities.

In proving our main theorem, the main technical step is showing that $J_{\infty} X \setminus g_{\infty}(J_{\infty} Y)$ has positive measure if $g$ is not \'etale. To this end, we have improved upon results of Kato and Kerz--Schmidt; see \cite[Lemma 3.5]{KatoWildRamificationRestrictionCurves}, \cite[Lemma 2.4]{KerzSchmidtOnDifferentNotionsOfTameness}. Our argument was inspired by those of Nakamura--Shibata in \cite{NakamuraShibataInversionOdAdjunctionForQuotientSingularities}. We let $\bD$ denote the formal disk $\Spec \kay \llbracket t \rrbracket$ and $\delta , \eta \in \bD$ denote its closed and generic points, respectively.

\begin{theoremA*}[\autoref{Cor.AbundanceNonLiftableArcs}]
Let $G$ be a finite group and $H\subset G$ be a subgroup. Let $g\colon Y\to X$ be a $G$-cover between normal varieties such that $Y$ is smooth and there is $y \in Y(\kay)$ being fixed by $H$. Let  $\bE\to \bD$ be another $G$-cover such that $H$ fixes some connected component $\bE_0$ of $\bE$ (i.e. $\bE_0 \to \bE$ is $H$-equivariant). Let $\sN \subset J_{\infty} X$ be the subset of arcs $\gamma \colon \bD\to X$ such that: $\gamma(\delta) = g(y)$,  $g$ is \'etale over $\gamma(\eta)$, and the pullback of $\gamma$ along $g$ induces $\bE\to \bD$. Then, $\sN$ has positive measure. 
\end{theoremA*}

We have discussed so far only the equal characteristic case, where we have a suitable theory of motivic integration. Motivic integration makes the proof of our Main Theorem rather formal. Unfortunately, such a theory is not available yet in mixed characteristics. Nevertheless, at least for surfaces, we may define the required stringy motives directly via minimal resolutions and we may rely on the strong factorization theorem. In \autoref{sec.Mixed_characteristic_surface case}, we prove our Main Theorem in this case as well without relying on motivic integration. See \autoref{thm.MainTheoremMixChar}. This lets us prove the finiteness of the \'etale fundamental group of log terminal surface singularities in all equal/mixed characteristics.

\begin{theoremB*}[\autoref{cor.FinitenessFUndGrouosMixChar}]
Let $(R,\fram,\kay)$ be a log terminal $2$-dimensional complete local ring with algebraically closed residue field $\kay$. Then, $\pi_1^{\mathrm{\Acute{e}t}}(\Spec R\setminus \{\fram\})$ is finite.
\end{theoremB*}

\subsection*{Acknowledgments} We thank Maciej Zdanowicz for all the technical support he kindly  provided to us regarding $W\sO$-rationality. We are also very thankful to Fabio Bernasconi who pointed out to us Kawamata's work \cite{KawamataIndex1CoversOfLogTerminalSurfaceSingularities} on index-$1$ covers of log terminal surface singularities. The authors are grateful to Richard Crew and Anurag Singh for answering some of our questions by email. The authors also thank Yusuke Nakamura for sending us a copy of the paper \cite{NakamuraShibataInversionOdAdjunctionForQuotientSingularities} which was very helpful in this work. The authors started working on this project during the second named author's visit to the Chair of Algebraic Geometry at the EPFL and during the first named author's visit to Tohoku University to attend the conference ``Singularities and Arithmetics,'' which was supported by JSPS KAKENHI Grant numbers 18K18710, 18H01112, and 16H06337. The authors are, respectively, very thankful to those institutions for their support and hospitality. Finally, we would like to thank the anonymous referees for their thoughtful feedback, which helped us improve this article. 

\begin{convention}
We fix an algebraically closed field $\kay$ of characteristic $p \geq 0$. Except for \autoref{sec.Mixed_characteristic_surface case}, every relative notion/object/operation is defined over $\kay$ unless otherwise explicitly stated. Also, $0 \in \bN$.
\end{convention}

\section{Stringy Motives} \label{sec.StringyMotivicInvaraints}
For the reader's convenience, we recall how to define the stringy motive $M_{\st}(X,\Delta)$ of a KLT log pair $(X,\Delta)$ as well as its main properties. We follow \cite{YasudaMotivicIntegrationDeligneMumfordStack}, where the reader can find full details. See also \cite[Chapter 6]{YasudaBook} for another reference that might be more friendly to non-specialists.

\subsection{Motivic integration}
We briefly review the theory of motivic integration over Deligne--Mumford stacks and in the equivariant setting. Motivic measures and integrals take values in a version of the Grothendieck ring of varieties. In this paper, we choose the one $\hat{\sM}'_{\kay,r}$ used in \cite{YasudaMotivicIntegrationDeligneMumfordStack}. The class of a variety $X$ in this ring is denoted by $\{X\}$ (square brackets $[\cdot]$ are reserved to express quotient stacks).
More generally, we can define the class $\{C\}$ of a constructible subset $C$ of a variety via its partition into locally closed subsets: if $C=\bigsqcup_{i=1}^nC_i$; with $C_i$ locally closed, then $\{C\}\coloneqq\sum_{i=1}^n\{C_i\}$. As usual, $\bL\coloneqq\{\bA^1\}$. The subscript $r \in \bN$ of $\hat{\sM}'_{\kay,r}$ means that the ring is adjoined with the fractional power $\bL^{1/r}$ of $\bL$. We often take a sufficiently divisible $r$ so that every rational number showing up sits in $\frac{1}{r}\bZ$. 

We work with the \emph{Poincar\'e realization} map $\hat{\sM}'_{\kay,r} \to \bZ\llparenthesis T^{-1/r}\rrparenthesis$, which sends $\{X\}$ to the Poincar\'e polynomial $ P(X)=P(X;T)$ of $X$; see \cite[\S8]{NicaiseTraceFormulaForVarietesOverADVF}. For example, $P(\bL)=T^2$. We define an order $>$ on $\bZ \llparenthesis T^{-1/r} \rrparenthesis$ by comparing coefficients lexicographically. That is, for two distinct elements $f,g\in \bZ \llparenthesis T^{-1/r} \rrparenthesis$, $f>g$ (resp. $f<g$) if and only if the leading coefficient (i.e., the coefficient of the highest degree term) of $f-g$ is positive (resp. negative). As usual, the
symbol $\le$ means either $<$ or $=$ (and likewise for $\ge$). 

\begin{proposition}\label{positive}
The Poincaré polynomial $P(X)$ of a variety $X$ is positive (i.e., $>0$ with respect to this
order) unless $X = \emptyset$. Moreover, for a countable family $\{X_i\}_{i \in \bZ}$ of varieties such that $X_i \ne \emptyset$ for some $i$, if a sum
$\sum_{i}P(X_{i})T^{e_{i}}$ with $e_{i}\in\frac{1}{r}\bZ$ is convergent in $\bZ \llparenthesis T^{-1/r} \rrparenthesis$, then it is positive.
\end{proposition}

\begin{proof}
Let $d$ be the dimension of $X$ and $c$ be the number of $d$-dimensional irreducible components of $X\otimes_{\kay} \bar{\kay}$. Then, $P(X)$ is a polynomial  of the form
\[
 c T^{2d} + (\text{terms of degree $<2d$}), 
\]
which shows the first assertion. 
As for the second assertion, from the assumption, at least one of the terms $P(X_i)T^{e_i}$ has degree $>-\infty$. Moreover, the convergence condition shows that the degrees of the terms are bounded above and the maximum is attained by finitely many terms, say $P(X_1)T^{e_1}, \dots, P(X_m)T^{e_m}$. 
Then, the leading coefficient of $P(X)$ is equal to the one of  $\sum_{i=1}^m P(X_i)T^{e_i}$, which is positive from the first assertion.
\end{proof}

\begin{definition}
We shall say that $\alpha \in \hat{\sM}'_{\kay,r}$ is \emph{positive} (i.e., $\alpha >0$) if $P(\alpha) >0$.
For two elements $\alpha, \beta \in \hat{\sM}'_{\kay,r}$, we write $\alpha > \beta$ if $P(\alpha)>P(\beta)$. 
\end{definition}

The positivity and the order defined above for elements of $\hat{\sM}'_{\kay,r}$ play an essential role in the proofs of our main results. Indeed, a key step in our arguments is to show $\alpha > \beta$ for  invariants $\alpha,\beta \in \hat{\sM}'_{\kay,r}$ associated with two singularities. This inequality is then proved by showing that $P(\alpha -\beta)$
is a countable sum of the form in Proposition \ref{positive}.

We denote the \emph{formal disk} by $\bD\coloneqq \Spec \kay\llbracket t \rrbracket $. We let $\delta,\eta \in \bD$ denote its closed and generic points; respectively. The \emph{punctured formal disk} is $\bD^* \coloneqq \Spec \kay \llparenthesis t \rrparenthesis=\bD\setminus \{\delta\}$. An \emph{arc} of a variety $X$ is a morphism $\bD \to X$. Given $n \in \bN$, an \emph{$n$-jet} of $X$ is a morphism $\Spec \kay [t]/(t^{n+1}) \to X$. The \emph{$n$-jet} scheme $J_n X$ of $X$ is the moduli scheme of $n$-jets of $X$, which is a separated scheme of finite type. For $n'\ge n$, there is a natural map $J_{n'} X\to J_n X$. The \emph{arc space} $J_\infty X$ is the moduli space of arcs of $X$ and is identified with the projective limit $\varprojlim J_n X$. For each $n \in \bN$, we have a \emph{truncation map} map $\pi _n \colon J_\infty X \to J_n X$. A subset $C \subset J_{\infty} X$ is said to be a \emph{stable} if there is $n \gg 0$ such that: $\pi_n(C)\subset J_nX$ is constructible, $C=\pi_n^{-1}\big(\pi_n(C)\big)$, and $\pi_{m+1}(C) \to \pi_m(C)$ is a piecewise trivial $\bA^{\dim X}_{\kay}$-bundle for all $m \geq n$. The arc space is equipped with the motivic measure $\mu_X$; for a stable subset $C\subset J_\infty X$, we write 
\begin{equation} \label{eqn.MeasureStableSubset}
\mu_X(C)\coloneqq \{\pi_n(C)\}\bL^{-n\dim X} \in \hat{\sM}'_{\kay,r},\quad n\gg 0.
\end{equation} 
Thus, we may say that a measurable subset $C\subset J_{\infty}X$ has
\emph{positive measure }if $P\big(\mu_{X}(C)\big)>0$.

If $X$ is endowed with an action of a finite group $G$ and if $C$ is a $G$-invariant stable subset, the same formula as in \autoref{eqn.MeasureStableSubset} gives an element of the $G$-equivariant version $G\textrm{-}\hat{\sM}'_{\kay,r}$ of $\hat{\sM}'_{\kay,r}$; which we denote by $\mu_X(C)$ too.
The ring $G\textrm{-}\hat{\sM}'_{\kay,r}$ is constructed from the Grothendieck ring of $G$-varieties (varieties given with a $G$-action) by the same procedure as the construction of $\hat{\sM}'_{\kay,r}$ from the Grothendieck ring of $\kay$-varieties. Using the forgetful map $G\textrm{-}\hat{\sM}'_{\kay,r} \to \hat{\sM}'_{\kay,r}$, the equivariant version of $\mu_X(C)$ maps to its non-equivariant version. However, there is a quotient map $G\textrm{-}\hat{\sM}'_{\kay,r} \to \hat{\sM}'_{\kay,r}$, which extends $\{X\} \mapsto \{X/G\}$ (e.g. for $X$ quasi-projective). The image of $\alpha \in G\textrm{-}\hat{\sM}'_{\kay,r}$ under this map is denoted by $\alpha/G$. This lets us define $\mu_X(C)/G$.

Let $G$ be a finite abstract group. A $G$-cover of $\bD$ is a morphism $\bE \to \bD$ together with a $G$-action on $\bE$ such that: $\bE$ is regular, $\bE\to \bD$ is flat of rank $\# G$, and 
$\bE^* \coloneqq \bD^* \times_\bD \bE \to \bD^*$ is a $G$-torsor. Equivalently, a $G$-cover $\bE\to\bD$ is the same as the normalization of $\bD$ along a $G$-torsor $\bE^{\ast} \to \bD^{\ast}$. \emph{A twisted formal disk} is the quotient stack $\sE=[\bE/G]$ associated to a $G$-cover $\bE\to \bD$, which is equipped with a morphism $\sE \to \bD$. 
Note that isomorphism classes of twisted formal disks are in one-to-one correspondence with isomorphism classes of Galois extensions of $\kay \llparenthesis t \rrparenthesis$ and that these are parametrized by an infinite dimensional space. See \cite[\S2.1]{YasudaPcyclicMackayCorrespondenceMotivicIntegration}, \cite[\S5]{YasudaMotivicIntegrationDeligneMumfordStack}.

Let $\sX$ be a separated, irreducible, and reduced Deligne--Mumford stack of finite type. A \emph{twisted arc} of $\sX$ is a representable morphism $\sE \to \sX$ from a twisted formal disk $\sE$. The moduli stack of twisted arcs of $X$ is denoted by $\sJ_{\infty} \sX$. For a twisted formal disk $\sE$, we denote by $\sJ^\sE_{\infty} \sX$ the locus of twisted arcs $\sE \to \sX$, so $\sJ _{\infty}\sX =\bigsqcup _{\sE} \sJ_{\infty}^\sE \sX$. An \emph{untwisted arc} of $\sX$ is a (necessarily representable) morphism $\bD \to \sX$. Since $\bD$ is a special case of twisted formal disks, untwisted arcs are twisted arcs. The moduli stack of untwisted arcs is denoted by $J_{\infty} \sX$, which is the substack $\sJ^\bD_{\infty}\sX$ of $\sJ_{\infty} \sX$.

To define the motivic measure $\mu_{\sX}$ on $\sJ_{\infty}\sX$, we need first to define the class $\{\sY\}\in \hat{\sM}'_{\kay,r}$ of a Deligne--Mumford stack $\sY$ of finite type. If $\sY$ is an algebraic space, it admits a partition $\sY=\bigsqcup_{i=1}^n\sY_i$ into schemes of finite type and we define $\{\sY\} \coloneqq \sum_{i=1}^n\{\sY_i\}$. For a Deligne--Mumford stack $\sY$, we define $\{\sY\}$ as the class of its coarse moduli space. We can also define the class $\{C\}$ of a constructible subset $C\subset \sY$ via its partition into locally closed subsets. If $C \subset \sJ_\infty \sX$ is a stable subset and if $\sJ_n \sX$ is a suitably defined stack of twisted $n$-jets with the truncation map $\pi_n \colon \sJ_\infty \sX\to \sJ_n \sX$, we define
\[
\mu_\sX(C)\coloneqq\{\pi_n(C)\}\bL^{-n\dim \sX},\quad n \gg 0,
\]
by mimicking \autoref{eqn.MeasureStableSubset}.

\begin{remark}
The above definition of $\mu_\sX$ is implicit in \cite{YasudaMotivicIntegrationDeligneMumfordStack}, which deals with schemes/stacks over $\Spf \kay\llbracket t\rrbracket$. To pass to this setting, we just base change from $\kay$ to $\Spf \kay\llbracket t \rrbracket$ and get a formal Deligne--Mumford stack over $\Spf \kay\llbracket t\rrbracket$, denoted again by $\sX$. To this formal Deligne--Mumford stack, we associate the untwisting stack $\mathrm{Utg}_\Gamma(\sX)^{\mathrm{pur}}$, which is a formal Deligne--Mumford stack over $\Spf \kay\llbracket t\rrbracket \times \Gamma$ for some Deligne--Mumford stack $\Gamma$ locally of finite type over $\kay$. 
The stack of twisted arcs of $\sX$ is then identified with the stack $J_\infty(\mathrm{Utg}_\Gamma(\sX)^{\mathrm{pur}})$ of untwisted arcs of $\mathrm{Utg}_\Gamma(\sX)^{\mathrm{pur}}$, that is, morphisms $\Spf \kay\llbracket t\rrbracket  \to \mathrm{Utg}_\Gamma(\sX)^{\mathrm{pur}}$ that are compatible with a morphism $\Spf \kay\llbracket t\rrbracket \to \Spf \kay\llbracket t\rrbracket \times \Gamma $ induced by a $\kay$-point of $\Gamma$. 
Then, $\mu_\sX$ was defined to be the motivic measure of the untwisted arc space $J_\infty(\mathrm{Utg}_\Gamma(\sX)^{\mathrm{pur}})$ of the untwisting stack $\mathrm{Utg}_\Gamma(\sX)^{\mathrm{pur}}$. In turn, the motivic measure of $J_\infty(\mathrm{Utg}_\Gamma(\sX)^{\mathrm{pur}})$ was defined in terms of truncation maps $J_\infty(\mathrm{Utg}_\Gamma(\sX)^{\mathrm{pur}})\to J_n(\mathrm{Utg}_\Gamma(\sX)^{\mathrm{pur}})$, similarly to the classical case where $\sX$ is a variety over a field. However, since  $\sJ_n \sX=J_n(\mathrm{Utg}_\Gamma(\sX)^{\mathrm{pur}})$ by definition, we may define the motivic measure on $\sJ_\infty \sX$ in the usual way in terms of truncation maps $\sJ_\infty \sX \to \sJ_n \sX$ eventually. 
\end{remark}

When $\sX$ is the quotient stack $[V/G]$ associated to an action of a finite group $G$ on a variety $V$, we then have $J_{\infty} \sX =[(J_{\infty}V)/G]$. 
To describe twisted arcs of $\sX$ in terms of the $G$-action on $V$, we need the notion of $G$-arcs. For a $G$-cover $\bE\to \bD$, an \emph{$\bE$-twisted $G$-arc} of $V$ is a $G$-equivariant morphism $\bE\to V$. Two $G$-arcs $\bE\to V$ and $\bE'\to V$ are isomorphic if there is a $V$-morphism $\bE\to \bE'$ that is an isomorphism as $G$-covers of $\bD$. We denote the space of $\bE$-twisted $G$-arcs of $V$ by 
\[
J_{\infty}^{\bE,G}V \coloneqq \{G\text{-equivariant }\bE\to V\}/ \Aut(\bE)^{\mathrm{opp}},
\]
where $\Aut(\bE)^{\mathrm{opp}}$ is the opposite group of the automorphism group of $\bE$ as a $G$-cover of $\bD$
and isomorphic to the centralizer $C_G(H)$ of the stabilizer $H$ of a connected component of $\bE$. 
In particular, for the trivial $G$-cover $\bE^{\mathrm{triv}}=\bD \times G \to \bD $, we have
$J_{\infty}^{\bE ^{\mathrm{triv}},G}V = (J_\infty V)/G$. Further, we have identifications $\sJ_{\infty}^\sE\sX = J_{\infty}^{\bE,G}V$, where $\sE=[\bE/G]$. Thus,
\[
\sJ_{\infty} \sX =\bigsqcup _{\sE} \sJ^\sE_\infty \sX= \bigsqcup _{\bE} J_\infty ^{\bE,G}V .
\]

\begin{remark}
We follow the convention that groups act on spaces from the left. The automorphism group $\Aut(\bE)\simeq C_G(H)^{\mathrm{opp}}$, \emph{a priori}, acts on $J_{\infty}^{\bE,G}V$ from the right. This induces the left action of $\Aut(\bE)^{\mathrm{opp}} \simeq C_G(H)$ on $J_{\infty}^{\bE,G}V$. This action is identical to the one induced by restricting the $G$-action on $V$ to $C_G(H)$.
\end{remark}

The measure $\mu_{\sX}$ restricted to $J_{\infty} \sX = [(J_\infty V)/G]$ and the measure $\mu_V$ are related as follows (with $\sX=[V/G]$ as above). For a measurable subset $C\subset J_\infty \sX$, its preimage $\tilde{C}\subset J_{\infty} V$ is a $G$-invariant measurable subset, and the following equality holds
\begin{equation} \label{eqn.MeasureOfQuotient}
\mu_\sX(C)=\mu_V(\tilde{C})/G.
\end{equation}

Let $\sY$ be another Deligne--Mumford stack satisfying the same conditions as $\sX$ and let $f\colon \sY \to \sX$ be a (not necessarily representable) morphism. For a twisted arc $\gamma\colon\sE \to \sY$, the composition $f\circ \gamma \colon \sY \to \sX$ is not generally a twisted arc as it may not be representable. Nevertheless, it factors uniquely as 
\[
 f\circ \gamma \colon \sE \to \sE' \xrightarrow{\gamma'} \sX,
\]
where $\sE \to \sE'$ is a $\bD$-morphism of twisted formal disks and $\gamma'$ is a twisted arc of $\sX$. 
If $\sX$ is an algebraic space, then $\sE'=\bD$ and $\gamma'$ are induced from the universality of the coarse moduli space $\sE\to \bD$. 
Sending $\gamma $ to $ \gamma'$ defines a map $f_\infty \colon \sJ_\infty \sY \to \sJ _\infty \sX$. When $f$ is proper and birational, $f_{\infty}$ is almost bijective (meaning bijective outside subsets of motivic measure zero); see \cite[Example 13.7]{YasudaMotivicIntegrationDeligneMumfordStack}. On the other hand, the restriction of $f_\infty$ to untwisted arcs $f_\infty |_{J_\infty \sY}\colon J_\infty \sY \to J_\infty \sX$ is not necessarily almost bijective, which is the main reason for introducing twisted arcs. When $\sY$ is the quotient stack $[V/G]$ and $\sX$ is the corresponding quotient scheme $V/G$, the map
\[
\sJ_\infty \sY = \bigsqcup_\bE J_\infty ^{\bE,G}V \to J_\infty (V/G)= \sJ_\infty \sX
\]
sends a $G$-arc $\bE\to V$ to the associated morphism of quotient schemes $\bD=\bE/G \to V/G$.

\begin{theorem}[The change of variables formula; {\cite[Theorem 16.1]{YasudaMotivicIntegrationDeligneMumfordStack}}]
Let $f\colon \sY \to \sX$ be a morphism of separated, irreducible, and reduced Deligne--Mumford stacks of finite type over $\kay$.
Let $h\colon f_{\infty}(A) \subset \sJ_\infty \sX \to \frac{1}{r}\bZ\cup\{\infty\}$ be a measurable function where $A \subset \sJ_{\infty}\sY$ is a measurable subset over which $f_{\infty}$ is almost geometrically injective. Then,
\[
\int _{f_\infty(A)} \bL^{h+\mathfrak{s}_\sX} \mathrm{d}\mu_\sX =
\int _{A} \bL^{h\circ f_\infty -\mathfrak{j}_f+\mathfrak{s}_\sY} \mathrm{d}\mu_\sY,
\]
where $\mathfrak{j}_f$ is the Jacobian order function associated to $f$ and $\mathfrak{s}_\sX$ and $\mathfrak{s}_\sY$ are the shift functions of $\sX$ and $\sY$; respectively.
\end{theorem}

Recall that the shift function $\mathfrak{s}_\sX$ is constantly zero on the space $J_\infty \sX$ of untwisted arcs. 

\subsection{Stringy motives of log pairs} \label{sec.StringyMotives} Let $X$ be a $d$-dimensional normal variety. If $x \in X$ is a closed point, $(J_{\infty}X)_{x} \subset J_{\infty}X$ denotes the preimage of $x$ along the projection $J_{\infty}X\to X$. Let $\Delta$ be a boundary on $X$ such that $(X,\Delta_X)$ has index $r$. Let $\sI_{X, \Delta}\subset\sO_{X}$ be defined by:
\[
\sI_{X,\Delta}\sO_X\big(r(K_X + \Delta)\big)=\Image\left(\left(\Omega_{X/\kay}^{d}\right)^{\otimes r}\to\sO_X\big(r(K_X + \Delta)\big)\right).
\]
In other words, $\sI_{X, \Delta}\subset\sO_{X}$ is the image of the pairing map:
\[
\left(\Omega_{X/\kay}^{d}\right)^{\otimes r} \otimes \sO_X\big(-r(K_X + \Delta)\big) \to \sO_X.
\]
Further, we consider the associated order function $\ord \sI_{X,\Delta}\colon J_{\infty}X\to\bN \cup \{\infty\}$ (namely, $\gamma \mapsto \length \kay\llbracket  t\rrbracket/\gamma^{-1} \sI_{X,\Delta}$) and set for notation ease
\[
F=F_{(X,\Delta)} = \frac{1}{r} \ord \sI_{X,\Delta}\colon J_{\infty}X\to \frac{1}{r} \cdot \bN\cup\{\infty\}.
\]
Note that $F^{-1}(\infty)=J_{\infty}Z$, where $Z = V(\sI_{X,\Delta})$, and that this is a subset of measure zero with respect to the measure $\mu_X$ on $J_{\infty} X$.
\begin{definition}[Stringy motive of a log pair]
Suppose that $(X,\Delta)$ is a KLT log pair that admits a log resolution. The \emph{stringy motive of $(X,\Delta)$} is defined as
\[
M_{\st}(X,\Delta) \coloneqq \int_{J_{\infty}X}\bL^{F_{(X,\Delta)}}\mathrm{d}\mu_{X}\in\hat{\sM}'_{\kay,r}.
\]
More generally, for a constructible subset $C\subset X$, the \emph{stringy motive of $(X,\Delta)$ along $C$}, denoted by $M_{\st}(X,\Delta)_C$, is defined as follows: let $(J_\infty X)_C:=(\pi_0)^{-1} (C) \subset J_\infty X$ and put 
\[
M_{\st}(X,\Delta)_C\coloneqq \int_{(J_{\infty}X)_{C}}\bL^{F_{(X,\Delta)}}\mathrm{d}\mu_{X}\in\hat{\sM}'_{\kay,r}
\]
If $X$ is log terminal, we define $M_{\st}(X)$ and $M_{\st}(X)_C$ as the corresponding stringy motives associated with the log pair $(X,0)$.
\end{definition}

We are mainly interested in the case where $C$ consists of a single closed point $x\in X$. In that case, we write $M_{\mathrm{st}}(X,\Delta)_C=M_{\mathrm{st}}(X,\Delta)_x$. 

\begin{remark} \label{rem.StringyMotivesLogResolution}
Since $(X,\Delta)$ has a log resolution, the KLT condition implies that the above integral converges. In fact, let $\phi\colon \tilde{X}\to (X,\Delta)$ be a log resolution and write:
\[
K_{\tilde{X}} \sim_{\bQ} \phi^*(K_X + \Delta) + \sum_{i\in I} b_i(E,X,\Delta) \cdot E_i,
\]
where the $E_i$ are prime divisors on $\tilde{X}$. For notation ease, we set
\[
K_{\tilde{X}/(X,\Delta)} \coloneqq \sum_{i\in I} b_i(E,X,\Delta) \cdot E_i \eqqcolon \sum_{i\in I} b_i \cdot E_i,
\]
If $r$ is the index of $(X,\Delta)$, then $b_i \in \frac{1}{r} \cdot \bZ$. Moreover, $(X,\Delta)$ being KLT means that the log discrepancies $a_i \coloneqq 1+b_i$ are positive. The change of variables formula and the explicit computation of stringy motives for simple normal crossing divisors yields
\begin{equation}\label{ex-formula}
M_{\st}(X,\Delta)_C=M_{\st}\Big(\tilde{X}, -K_{\tilde{X}/(X,\Delta)}\Big)_{\phi^{-1}(C)} = \sum_{J \subset I} \{E_J^{\circ}\cap \phi^{-1}(C)\} \prod_{j\in J} \frac{\bL - 1}{\bL^{a_j}-1},
\end{equation}
where $ E^{\circ}_J \coloneqq \bigcap_{j\in J} E_j \big\backslash \bigcup_{j\notin J} E_j$.
These formulas are well-known to specialists. They can be found in \cite[Proposition 8.4]{YasudaToward}, \cf \cite[Proposition 3.4.4 and Theorem 4.1.2]{ChambertLoirNicaiseSebag}. Further details are given in \cite{YasudaBook}.

The invariant $M_{\st}(X,\Delta)_C$ satisfies the following additivity property: when $C_1,\dots,C_n\subset X$ are mutually disjoint constructible subsets, then 
\begin{equation}\label{additive}
M_{\st}(X,\Delta)_{\bigsqcup _{i=1}^n C_i}=\sum_{i=1}^n M_{\st}(X,\Delta)_{C_i},
\end{equation}
which is a direct consequence of either the definition or formula \autoref{ex-formula}.
Also, one readily sees that $M_{\st}(X,\Delta)_x$ depends only on the formal completion (or the henselization) of $X$ at $x$ and so we may regard it as an invariant of $\hat{\sO}_{X,x}$ (or $\sO_{X,x}^{\mathrm{h}}$). 
\end{remark}

\begin{example}[Surfaces and minimal resolutions] \label{ex.Surfaces}
 With notation as above, suppose that $(x,X)$ is a log terminal surface singularity (with Gorenstein index $r \in \bN$) and $\phi \colon \Tilde{X} \to (x,X)$ is a minimal log resolution, so that $a_i \in (0,1]$ (see \cite[Claim 2.26.4, p. 56]{KollarSingulaitieofMMP}). In this case, $(x,X)$ is further rational and so the dual graph $\Gamma$ associated to $\phi \colon \Tilde{X} \to (x,X)$ is a tree of $\bP^1$'s, which means that $E_i \cong \bP^1$ for all $i \in I$. Let us recall that the set of vertices of $\Gamma$ is $I$ and there is an edge connecting two different vertices $i,j \in I$ if and only if $E_i \cap E_j \neq \emptyset$. Let $H$ be the set of edges of $\Gamma$. We denote an element of $H$ by $[i,j]$ where $i,j\in I$ are the vertices that edge connects (of course, $[i,j]=[j,i]$). Further, denote by $m_i \geq 1$ the number of edges sticking out of a vertex $i\in I$. From \autoref{ex-formula}, we may compute $M_{\st}(X)_x$ as
\begin{align*}
    M_{\st}(X)_x = &\sum_{i\in I}\left\{E_{i} \Bigg\backslash \bigcup_{j \neq i} E_j \right\}\frac{\bL-1}{\bL^{a_{i}}-1}+\sum_{\{i,j\}\subset I}\{E_{i}\cap E_{j}\}\frac{(\bL-1)^{2}}{(\bL^{a_{i}}-1)(\bL^{a_{j}}-1)}\\
    = &\sum_{i\in I} (\bL +1 - m_i) \frac{\bL-1}{\bL^{a_{i}}-1}+\sum_{[i,j]\in H}\frac{(\bL-1)^{2}}{(\bL^{a_i}-1)(\bL^{a_j}-1)} \in \bZ\llparenthesis\bL^{-1/r} \rrparenthesis \subset \hat{\sM}'_{\kay,r}.
\end{align*}
We will further review how the first equality above works in the proof of \autoref{lem.SurfaceCaseDescription} below. Its Poincar\'e realization then is
\[
P_{\st}(X)_x =  \sum_{i\in I} (T^2 +1 - m_i) \frac{T^2-1}{T^{2a_{i}}-1}+\sum_{h\in H}\frac{(T^2-1)^{2}}{(T^{2a_i}-1)(T^{2a_j}-1)} \in \bZ\llparenthesis T^{-2/r} \rrparenthesis,
\]
which is a formal Laurent series in $T^{-2/r}$ with integral coefficients.
\end{example}

\begin{example}[Cones] \label{ex.ConesOverFanos}
Let $V$ be a \emph{smooth} Fano variety and write $K_V + aL \sim_{\bQ} 0$ for some ample Cartier divisor $L$ on $V$ and some $0<a \in \bQ$. Then, the affine cone $X=\Spec \bigoplus_{i \in \bN} H^0(V,iL)$ is log terminal; let $x \in X$ denote its vertex. See \cite[Lemma 3.1]{KollarSingulaitieofMMP}. The blowup $Y \coloneqq\Bl_x X \to X$ is a resolution with exceptional divisor $E \cong V$. 
From \eqref{ex-formula}, it follows that:
\[
M_{\st}(X)= (\bL-1)\{V\} + \{V\} \frac{\bL-1}{\bL^a-1} \text{ whereas } M_{\st}(X)_x = \{V\} \frac{\bL-1}{\bL^a-1}.
\]
Observe that $\{X\}=(\bL-1)\{V\}+1$. For instance, if $V$ were a smooth hypersurface of degree $d$ inside $\bP^n$ (with $L$ being the hyperplane section), then $a= n+1-d$. For example, if $V \subset \bP^d$ is the rational normal curve, then $\{V\}=\bL+1$ and $a=2/d$ so that $M_{\st}(X)_x=(\bL^2-1)/(\bL^{2/d}-1)$.

The formula above can be generalized to the case in which $V$ is a \emph{log terminal} Fano variety as follows:
\[
 M_{\mathrm{st}}(X)_x = M_{\mathrm{st}}(V)\frac{\bL-1}{\bL^{a+1} -1 }. 
\]
To show this, let us take a log resolution $\tilde{V}\to V$. Since $Y$ is a line bundle over $V$, we have the corresponding log resolution $f\colon\tilde{Y} \to Y $. Let us write $K_{\tilde{V}/V} = \sum _{i\in I} b_i E_i$ and  $K_{\tilde{Y}/Y} = \sum _{i\in I} b_i F_i$, where the $E_i$ are prime divisors on $\tilde{V}$ and the $F_i$ are their corresponding prime divisors on $\tilde{Y}$; respectively. Then 
\[
M_{\mathrm{st}}(V)=\sum_{J\subset I} \{E_J^\circ\} \prod_{j \in J} \frac{\bL-1}{\bL^{b_j+1}-1}. 
\]
To compute $M_{\mathrm{st}}(X)_x$, observe that
\[K_{\tilde{Y}/X}
= K_{\tilde{Y}/Y} +  f^* K_{Y/X} = \sum_{i\in I} b_i F_i + a G,
\]
where $G$ denotes the strict transform of the exceptional divisor of $Y\to X$, which is isomorphic to $V$. Setting $F_0 \coloneqq G$ and $b_0 \coloneqq a$, we may write $K_{\tilde{Y}/X} = \sum _{i\in I\cup \{0\}} b_i F_i $. Thus, letting $g$ denote the map $\tilde{Y} \to Y \to X$, we have:
\begin{align*}
    M_{\mathrm{st}}(X)_x 
    &= \sum_{J\subset I\cup \{0\}} \big\{F_J^\circ \cap g^{-1}(x)\big\} \prod_{j\in J}\frac{\bL-1}{\bL^{b_j+1}-1} \\
    &= \sum_{J\subset I} \big\{F_J^\circ \cap g^{-1}(x)\big\} \prod_{j\in J}\frac{\bL-1}{\bL^{b_j+1}-1} 
    + \sum_{0\in J\subset I\cup \{0\}} \big\{F_J^\circ \cap g^{-1}(x)\big\} \prod_{j\in J}\frac{\bL-1}{\bL^{b_j+1}-1} . 
\end{align*}
Since $g^{-1}(x)=F_0$, we have that $F_J^\circ \cap g^{-1}(x) =\emptyset$ for every $J\subset I$. Therefore,
\begin{align*}
    M_{{\mathrm{st}}}(X)_x 
    &=   \sum_{0\in J\subset I\cup \{0\}} \big\{F_J^\circ \cap g^{-1}(x)\big\} \prod_{j\in J}\frac{\bL-1}{\bL^{b_j+1}-1} \\
    &=  \sum_{J'\subset I} \{E_{J'}^\circ \} \prod_{j\in J'\cup\{0\}}\frac{\bL-1}{\bL^{b_j+1}-1} \\
    &=  \left(\sum_{J'\subset I} \{E_{J'}^\circ \} \prod_{j\in J'}\frac{\bL-1}{\bL^{b_j+1}-1}\right) \frac{\bL-1}{\bL^{b_0+1}-1}\\
    & = M_{\mathrm{st}}(V)\frac{\bL-1}{\bL^{b_0+1}-1};
\end{align*}
as required. This concludes our example about stringy motives of cone singularities.
\end{example}

If a finite group $G$ acts on $X$  preserving the boundary $\Delta$, there exists a canonical lift of $M_{\st}(X,\Delta)$ to $G\textrm{-}\hat{\sM}'_{\kay,r}$, which we denote by $M_{\st}(X,\Delta)$ by abuse of notation. The same lifting principle holds for $M_{\st}(X,\Delta)_x$ if $G$ fixes $x$. This lets us define $M_{\st}(-)/G$ where ``$-$'' denotes any of the data we had considered above. Once again, the invariant $M_{\st}(X,\Delta)_x/G$ depends only on the completion/henselization of $X$ at $x$ (and the $G$-action on it). 

We may also define $M_{\st}(-)/G$ in terms of motivic integration over Deligne--Mumford stacks. This interpretation will be useful when a change of variables formula is required. Let $\sX$ be the quotient stack $[X/G]$. Then, $M_{\st}(X/G) =M_{\st}(\sX)$ (use \cite[Theorem 1.2]{YasudaMotivicIntegrationDeligneMumfordStack}) whereas $M_{\st}(X)/G$ is equal to the untwisted stringy motive $M_{\st}^{\utd}(\sX)$
\cite[Definition 13.2]{YasudaMotivicIntegrationDeligneMumfordStack}, which is
by definition a motivic integral over the space $J_{\infty}\sX$ of untwisted arcs $\bD \to\sX$. That is:
\[
M_{\st}(X)/G =\left ( \int_{J_{\infty} X} \bL^{F_{X}} \mathrm{d} \mu_{X}\right)\Biggm/G= \int_{J_{\infty} \sX} \bL^{F_{\sX}} \mathrm{d} \mu_{\sX} \eqqcolon M_{\st}^{\utd}(\sX),
\]
where $F_{\sX}$ is the function corresponding to the $G$-invariant function $F_{X}$ by identifying $J_\infty \sX$ with $(J_\infty X)/G$. 
Analogous descriptions apply to $M_{\st}(X,\Delta)/G$ and $M_{\st}(X,\Delta)_x/G$, e.g.
\[
M_{\st}(X,\Delta)_x/G = M_{\st}^{\utd}(\sX, \bar{\Delta})_{\bar{x}} \coloneqq \int_{(J_{\infty}\sX )_{\bar{x}} } \bL^{F_{\sX, \bar{\Delta}}} \mathrm{d} \mu_{\sX}
\]
where $\bar{\Delta}$ and $\bar{x}$ are; respectively, the divisor and point of $\sX$ induced by $\Delta$ and $x$. Further, observe that the stringy motive  $M_{\st}(\sX,\bar{\Delta})_{\bar{x}}$ (without the
adjective ``untwisted'') is instead a motivic integral over the
whole space $\sJ_{\infty}\sX$ of twisted arcs of $\sX$:
\[
M_{\st}(\sX, \bar{\Delta})_{\bar {x}} = \int_{(\sJ_\infty \sX)_{\bar{x}}} \bL^{F_{\sX,\bar{\Delta}}+ \mathfrak{s}_\sX} \mathrm{d}\mu_\sX = M_{\st}(X/G, \Delta/G)_{\bar{x}},
\]
where the function $F_{\sX,\bar{\Delta}}$ (as defined above) extends to $\sJ_\infty \sX$, $\mathfrak{s}_{\sX}$ is the shift function, $\Delta/G$ and $x/G$ are respectively the divisor and the point on $X/G$ induced from $\Delta$ and $x$, and the second equality is a direct application of \cite[Theorem 1.2]{YasudaMotivicIntegrationDeligneMumfordStack}. Those functions take values in $\frac{1}{r}\bZ$ except that $F_{\sX,\bar{\Delta}}$ takes value $\infty$
on the twisted arc space $\sJ_\infty \sZ$ of a proper closed substack $\sZ \subsetneq \sX$. However, $\sJ_\infty \sZ$ has measure zero as a subset of $\sJ_\infty \sX$. Note also that $\mathfrak{s}_\sX$ is identically zero on $J_\infty \sX$. Summing up, $M_{\st}^{\utd}(\sX, \bar{\Delta})_{\bar x}$ is obtained from $M_{\st}(\sX, \bar{\Delta})_{\bar x}$ by restricting the domain of integration to untwisted arcs.  Therefore,
\[
M_{\st}(X/G, \Delta/G)_{\bar{x}} = M_{\st}(\sX, \bar{\Delta})_{\bar x} \ge M_{\st}(\sX, \bar{\Delta})_{\bar x}^{\utd}
\]
and the inequality is strict whenever  $\sJ_\infty \sX\setminus J_\infty \sX$ has positive measure.
\begin{setup} \label{setup.MainTheorems}
Let us consider a commutative diagram
\[
\xymatrix{
(X_1,\Delta_1)  \ar[dr]_-{g_1}& & (X_2,\Delta_2) \ar[ll]_-{f} \ar[dl]^-{g_2}\\
&(X_0,\Delta_0)&
}\]
of crepant finite covers among log pairs. Moreover, suppose that $g_i$ is generically Galois with Galois group $G_i$. Let us define $\sX_i\coloneqq[X_i/G_i]$ and let $\bar{\Delta}_i$ be the boundary divisor on $\sX_i$ induced by $\Delta_i$. Let $x_0$ be a $\kay$-point of $X_0$ and let $\bar{x}_i$ be its unique lift to $\sX_i$.
\end{setup}

\begin{proposition}[{\cite[Theorem 16.2]{YasudaMotivicIntegrationDeligneMumfordStack}}]
Work in \autoref{setup.MainTheorems}. Then, for all $i \in \{1,2\}$:
 \[
 M_{\st}(X_0,\Delta_0)=M_{\st}(\sX_i,\bar{\Delta}_i)
\text{ and likewise }
M_{\st}(X_0,\Delta_0)_{x_0}=M_{\st}(\sX_i,\bar{\Delta}_1)_{\bar{x}_i}.
 \]
\end{proposition}

\begin{corollary} \label{cor.KeyCorollary}
Work in \autoref{setup.MainTheorems}. Then, 
\[
M_{\st}(X_1,\Delta_1)/G_1 
=
M_{\st}^{\utd}(\sX_1,\bar{\Delta}_1)
\ge 
M_{\st}^{\utd}(\sX_2,\bar{\Delta}_2)
=
M_{\st}(X_2,\Delta_2)/G_2
\]
Moreover, if $J_\infty X_1 \setminus f_\infty (J_\infty X_2)$ has positive measure, then the inequality is strict. 

Let $x_i\in X_i$ be the preimage of $x_0$ (which we are assuming to be a point and so that $G_i$ fixes $x_i$). Then,
\[
M_{\st}(X_1,\Delta_1)_{x_1}/G_1 
=
M_{\st}^{\utd}(\sX_1,\bar{\Delta}_1)_{\bar{x}_1} 
\ge  
M_{\st}^{\utd}(\sX_2,\bar{\Delta}_2)_{\bar{x}_2}
=
M_{\st}(X_2,\Delta_2)_{x_2}/G_2.
\]
Moreover, the inequality is strict if $(J_\infty X_1)_{x_1} \setminus f_\infty ((J_\infty X_2)_{x_2})$ has positive measure.
\end{corollary}

\begin{proof}
From the change of variables formula,
\[
M_{\st}^{\utd}(\sX_1,\bar{\Delta}_1)
-
M_{\st}^{\utd}(\sX_2,\bar{\Delta}_2)
= \int_{J_\infty \sX_1 \setminus f_\infty(J_\infty  \sX_2)}\bL^{F_{\sX_1,\bar{\Delta}_1} +\mathfrak{s}_{\sX_1}}\mathrm{d}\mu_{\sX_1}\ge 0.
\]
The inequality is strict if $J_\infty \sX_1 \setminus f_\infty(J_\infty  \sX_2)$ has positive measure. By \autoref{eqn.MeasureOfQuotient}, this is the case if $J_\infty X_1 \setminus f_\infty (J_\infty X_2)$ has positive measure. The local statements are shown likewise. 
\end{proof}
\subsection{Explicit description for log terminal surface singularities} \label{sec.ExplicitDescriptionSurfaceCase}

In this section, we compute $M_{\st}(X)_x/G$ explicitly for a log terminal surface singularity $x\in X$ with a given action by a finite group $G$ fixing $x$. Our results will be analogous to those in \autoref{ex.Surfaces} but we will need to use $G$-equivariant minimal log resolutions instead. 

Recall that if $(X,\Delta)$ is a KLT surface pair then $X$ is necessarily $\bQ$-factorial and so log terminal; see \cite[Corollary 4.11]{TanakaMMPExcellentSurfaces}, \cite[Corollary 2.35]{KollarMori}. Thus, since we are ultimately interested in the topology around $x \in X$, there is no harm in assuming $\Delta = 0$. This will turn out being a considerable simplification on the possible shapes of the dual graph associated to the minimal resolution (i.e., in the terminology of\cite[\S 2.2]{KollarSingulaitieofMMP}, we shall avoid dealing with \emph{extended dual graphs}). Further, we may shrink $X$ such that $X\setminus \{x\}$ is nonsingular. In such case:
\begin{equation} \label{eqn.GlobalvsLocalSurfaceCase}
M_{\st}(X)/G = \big\{\big(X \setminus \{x\}\big)/G \big\} + M_{\st}(X)_x/G. 
\end{equation}
More generally, if $X \setminus \{x_1,\ldots, x_k\}$ is regular and $G$ acts on $\{x_1, \ldots, x_k\}$ then from \autoref{additive},
\[
M_{\st}(X)= \bigl\{X\setminus \{x_1,\ldots, x_k\}\bigr\} + \sum_{i=1}^k M_{\st}(X)_{x_i}
\]
and so
\begin{equation} \label{eqn.GlobalvsLocalSurfaceCaseManyPoints}
M_{\st}(X)/G= \bigl\{\bigl(X\setminus \{x_1,\ldots, x_k\}\bigr)/G\bigr\} + \sum_{j=1}^l M_{\st}(X)_{x_j}/\Stab(x_j), 
\end{equation}
where the $x_1, \ldots, x_l$ are representatives for the orbits of the action of $G$ on $\{x_1, \ldots, x_k\}$; in particular $\{x_1,\ldots, x_k\} = \bigsqcup_{j=1}^l  G x_j$.

Let us commence by making a general remark on how a $G$-equivariant log resolution can be used to compute $M_{\mathrm{st}}(X)_x/G$.
Let us take a $G$-equivariant log resolution $\phi\colon \tilde{X}\to (x,X)$ and use the notation of \autoref{rem.StringyMotivesLogResolution} (for $I$, $E_i$, $a_i$, and so on). Following \cite[Definition 5.1]{batyrevNonArchimedeanIntegralsStringy1999}, we introduce the following notion:
\begin{definition}
We say that $\phi$ is \emph{$G$-normal} if, for every node $w\in E$,
the stabilizer $\Stab(w)$ preserves each of the two irreducible components
$E_{i}$ including $w$. In terms of the dual graph $\Gamma$ of $E$, this means that if an
element $g\in G$ fixes some edge $e$ then it fixes the two vertices adjoint to $e$ (instead of swapping them). 
\end{definition}

\emph{Let us suppose that $\phi$ is $G$-normal}. Then, for all $\iota \in I/G$, the prime divisors
$\{E_{i}\}_ {i\in\iota}$ are mutually disjoint. Here, we think of $\iota$ as an orbit of $G$ acting on $I$. Thus, we may define 
\[
E_{\iota}\coloneqq\bigsqcup_{i\in\iota}E_{i}\text{ and }
E_{\iota}^{\circ}\coloneqq E_{\iota} \Biggm\backslash \bigcup_{\kappa\in(I/G)\setminus\{\iota\}}E_{\kappa}.
\]
Observe that, for all $\iota \in I/G$, the prime divisors $\{E_{i}\}_{i\in\iota}$ have the same log discrepancy; which we denote by $a_{\iota}$ (i.e., $a_{\iota} = a_i$ for all $i \in \iota$). Our general remark is the following:

\begin{lemma} \label{lem.SurfaceCaseDescription}
With notation and hypotheses as above, the following formula holds
\begin{align*}
M_{\st}(X)_x & /G=\sum_{\iota\in I/G}\{E_{\iota}^{\circ}/G\}\frac{\bL-1}{\bL^{a_{\iota}}-1}+\sum_{\{\iota,\kappa\}\subset I/G}\{(E_{\iota}\cap E_{\kappa}) /G\}\frac{(\bL-1)^{2}}{(\bL^{a_{\iota}}-1)(\bL^{a_{\kappa}}-1)},
\end{align*}
where $\{\iota,\kappa\}$ runs over the subsets of $I/G$ with two elements.
\end{lemma}
\begin{proof}
This is basically the $g=1$ part of the formula defining the orbifold
$E$-function \cite[Definition 6.3]{batyrevNonArchimedeanIntegralsStringy1999}. 
Note also that the formula in the case $G=1$ is well-known and is a special case of \eqref{ex-formula}. Let us first recall how the formula in that case is deduced:
\begin{align*}
M_{\st}(X)_x 
& =\int_{(J_{\infty}X)_x}\bL^{F_{(X,0)}}\mathrm{d}\mu_{X}\\
& =\int_{(J_{\infty}\tilde{X})_{\phi^{-1}(x)}}\bL^{-\ord K_{\tilde{X}/X}}\mathrm{d}\mu_{\tilde{X}}\\
& =\sum_{\substack{i\in I\\ n\in \bN}}\{E_{i}^{\circ}\}(\bL-1)\bL^{-a_{i}n}+\sum_{\substack{\{i,j\}\subset I\\ n,m \in \bN }}\{E_{i}\cap E_{j}\}(\bL-1)^{2}\bL^{-na_{i}-ma_{j}}\\
& =\sum_{i\in I}\{E_{i}^{\circ}\}(\bL-1)\sum_{n\ge 0}\bL^{-a_{i}n}+\sum_{\{i,j\}\subset I}\{E_{i}\cap E_{j}\}(\bL-1)^{2}\sum_{n,m \geq 0}\bL^{-na_{i}-ma_{j}}\\
& =\sum_{i\in I}\{E_{i}^{\circ}\}\frac{\bL-1}{\bL^{a_{i}}-1}+\sum_{\{i,j\}\subset I}\{E_{i}\cap E_{j}\}\frac{(\bL-1)^{2}}{(\bL^{a_{i}}-1)(\bL^{a_{j}}-1)}.
\end{align*}
The second equality follows from the change of variables formula. For the third one, observe that the
term $\{E_{i}^{\circ}\}(\bL-1)\bL^{-a_{i}n}$ is the contribution
of arcs on $\tilde{X}$ meeting $E_{i}$ with order $n$ but
not meeting any other exceptional prime divisor: $\{E_{i}^{\circ}\}(\bL-1)\bL^{-n}$
is the measure of the set of those arcs (see \cite[Chapter 7, Lemma 3.3.3]{ChambertLoirNicaiseSebag}; note that our normalization of the motivic measure is different from the one in the cited reference by the factor of $\bL$ to the power of the dimension of the variety in question)
and since $b_i = a_i -1$ is the multiplicity of $E_i$ in $K_{\tilde{X}/X}$, then $\bL^{-b_{i}n}$ is the
value of the function $-\ord K_{\tilde{X}/X}$ there. Likewise, the
term $\{E_{i}\cap E_{j}\}(\bL-1)^{2}\bL^{-na_{i}-ma_{j}}$ is the
contribution of arcs that meet $E_{i}$ and $E_{j}$ with orders
$n$ and $m$ respectively: $\{E_{i}\cap E_{j}\}(\bL-1)^{2}\bL^{-2n}$ is the
measure of the set of those arcs  (again \cite[Chapter 7, Lemma 3.3.3]{ChambertLoirNicaiseSebag}) and $\bL^{-nb_{i}-mb_{j}}$ is the
value of the function $-\ord K_{\tilde{X}/X}$ there. The remaining two equalities are simple algebraic manipulations.

We now explain how we modify the above computation to get the formula of the lemma for a general finite group $G$. 
Let $\iota\in I/G$ and consider the set $C_{n}$ (resp. $C_{\ge n}$)
of arcs on $\tilde{X}$ that meet $E_{i}$ for some $i\in\iota$ with order $n$
(resp. $\ge n$) but do not meet any exceptional prime divisor (with index) not
belonging to $\iota$. These are cylinders of level $n$ and their
images $\bar{C}_n$ and $\bar{C}_{\ge n}$ in the $n$-th
jet scheme $J_{n}\tilde{X}$ are; respectively, a $\mathbb{G}_{\mathrm{m}}$-bundle and an $\bA^{1}$-bundle
over $E_{\iota}^{\circ}$. Therefore, 
\[
\{\bar{C}_{\ge n}/G\}=\{E_{\iota}^{\circ}/G\}\bL \text{ and } \left\{ \left(\bar{C}_{\ge n}-\bar{C}_n\right)/G\right\} =\{E_{\iota}^{\circ}/G\}
 \]
by the relation imposed in the definition of our complete Grothendieck
ring of varieties \cite[Definitions 9.5 and 9.6]{YasudaMotivicIntegrationDeligneMumfordStack}.
Hence, 
\[
\{\bar{C}_n/G\}=\{E_{\iota}^{\circ}/G\}(\bL-1).
\] 
This explains the first summation on the right-hand side in the formula of the lemma. The second summation is explained likewise.
\end{proof}

\subsubsection{On the $G$-normality of the minimal resolution of a surface}
\autoref{lem.SurfaceCaseDescription} raises the question of whether a $G$-normal log resolution exists. We discuss next the $G$-normality of the minimal resolution of $X$, say $\psi\colon X'\to (x,X)$; see \cite[Theorem 2.25]{KollarSingulaitieofMMP}. By
\cite[p. 123]{KollarSingulaitieofMMP} (also see \cite[\S3]{KollarFlipsAndAbundance}), the exceptional set of $\psi$ is a simple normal crossing divisor whose dual graph $\Gamma'$
is either a straight line (Figure \ref{fig:straight}) or has three
straight branches sticking out of a single vertex (Figure \ref{fig:three-branches}). In the latter case, we refer to the single vertex as the \emph{node.}
\begin{figure}[H]
\centering
\begin{minipage}{.5\textwidth}
  \centering
  \[
\xymatrix{\circ\ar@{-}[r] & \circ\ar@{-}[r] & \cdots & \circ\ar@{-}[l]}
\]
  \captionof{figure}{Straight dual graph.}
  \label{fig:straight}
\end{minipage}%
\begin{minipage}{.5\textwidth}
  \centering
  \[
\xymatrix{\circ\ar@{-}[r] & \cdots & \circ\ar@{-}[l]\ar@{-}[r]\ar@{-}[d] & \cdots & \circ\ar@{-}[l]\\
 &  & \vdots\\
 &  & \circ\ar@{-}[u]
}
\]
  \captionof{figure}{Three branches dual graph.}
  \label{fig:three-branches}
\end{minipage}
\end{figure}

We study the $G$-normality of $\psi$ by checking
three distinct cases separately:

\emph{The case where $\Gamma'$ is a straight line and has an odd number
of vertices:} In this case, $\Gamma'$ has only two automorphisms; namely,
the identity and the involution switching the two ends. The involution
does not fix any edge. Thus, the minimal resolution $\psi$ is $G$-normal (for all $G$)
and we choose it as our $G$-normal log resolution $\phi\colon\tilde{X}\to X$.

\emph{The case where $\Gamma'$ is a straight line and has an even
number of vertices:} As in the previous case, $\Gamma'$ has only two automorphisms. This time, however, the involution does fix the
middle edge and switches the two vertices adjacent to it. Hence, $\psi$ may not be $G$-normal. To fix it, we take the
blowup $\tilde{X}\to X'$ at the point corresponding to the middle
edge. The dual graph associated to the resulting log resolution $\phi\colon\tilde{X}\to X$
is a straight line with an odd number of vertices. As before, $\phi$ is $G$-normal and we choose it as our $G$-normal log
resolution. 

\emph{The case where $\Gamma'$ has three branches:} Any automorphism
of $\Gamma'$ fixes the node. If it fixes some edge $e$, then it
fixes the branch including $e$. Since one end of the branch; the
node, is fixed, the automorphism fixes all the edges and the vertices
in the branch. In particular, it fixes the two vertices adjacent to
$e$. This shows that the minimal resolution $\psi\colon X'\to X$
is $G$-normal. We choose $\psi$ to be our log resolution $\phi\colon\tilde{X}\to X$. 

Summing up, we have constructed a $G$-normal log resolution $\phi\colon\tilde{X}\to X$
in each of the above three cases, which is the minimal resolution possibly
followed by the blowup at a $G$-fixed point. Furthermore, it is independent of the group action and is minimal among $G$-normal log resolutions.
\begin{definition}
We refer to $\phi \colon \tilde{X} \to (x,X)$ as the \emph{modified minimal resolution
of $x\in X$}.
\end{definition}
According to \cite[Claim 2.26.4, p. 56]{KollarSingulaitieofMMP},
an exceptional prime divisor of the minimal resolution $\psi\colon X'\to X$
has log discrepancy $\leq 1$. In the case where $\tilde{X}$ is a one-point
blowup of $X'$, then the exceptional divisor of this blowup has log
discrepancy $\leq 2$ over $X$. Thus:
\begin{lemma}
\label{lem:log discrep at most 1}All but one exceptional prime divisors
of the modified minimal resolution have log discrepancy $\le1$. Moreover,
the possible exception has log discrepancy $\le2$. 
\end{lemma}

With the above in place, we can now specialize the computation of $M_{\st}(X)_x/G$ in \autoref{lem.SurfaceCaseDescription} to the case where $\phi$ is the modified minimal resolution. Let $\Gamma$ be the dual graph associated to the modified minimal resolution $\phi\colon\tilde{X}\to X$. Note that $\Gamma$ also has one of the forms of either \autoref{fig:straight} or \autoref{fig:three-branches} and the same is true for the quotient $\Gamma/G$.  We may think of a vertex $\iota \in I/G$ as a set of orbits of vertices of $I$. Note that two vertices $i, j \in \iota$ correspond to disjoint exceptional prime divisors (precisely because the modified minimal resolution is $G$-normal). Further, the log discrepancies are constant across $G$-orbits, i.e., $a_i=a_j$ if $i,j \in \iota$. This let us define the quotient log discrepancy $a_{\iota}$ of $\iota \in I/G$. An edge $[\iota,\kappa] \in H/G$ of $\Gamma/G$ corresponds to at least one vertex $i \in \iota$ being connected to a vertex in $j \in \kappa$ (by an edge $[i,j] \in H$). Note that the set of edges $\{[i,j] \in H \mid i\in \iota, j\in \kappa\}$ corresponds (bijectively) to the intersection points of $E_{\iota} \coloneqq \bigsqcup_{i \in \iota} E_i$ and $E_{\kappa} \coloneqq \bigsqcup_{j \in \kappa} E_j$, which get all identified under the action of $G$. Thus, an edge $[\iota,\kappa]$ of $\Gamma/G$ corresponds to $E_{\iota} \cap E_{\kappa} \neq \emptyset$ and $\{(E_{\iota} \cap E_{\kappa})/G\} = \delta_{\iota, \kappa} \in \hat{\sM}'_{\kay,r}$ (where $\delta_{\iota,\kappa}$ is Kronecker's delta) as there is only one edge connecting two vertices in $\Gamma/G$. Further, we see that $\{(E_{\iota} \setminus \bigcup_{\kappa \neq \iota} E_{\kappa})/G\}\in \hat{\sM}'_{\kay,r}$ equals $\bL+1 - m_{\iota}$ where $m_{\iota}$ is the number of edges of $\Gamma/G$ sticking out of $\iota$---this uses L\"uroth's theorem to see that $(E_{\iota} \setminus \bigcup_{\kappa \neq \iota} E_{\kappa})/G$ are rational curves.

In conclusion:
\begin{equation} \label{eqn.DescriptionOfStringyInvariantsForCurves}
M_{\st}(X)_x/G = \sum_{\iota\in I/G} (\bL +1 - m_{\iota}) \frac{\bL-1}{\bL^{a_{\iota}}-1}+\sum_{[\iota,\kappa]\in H/G}\frac{(\bL-1)^{2}}{(\bL^{a_{\iota}}-1)(\bL^{a_{\kappa}}-1)} \in \bZ\llparenthesis\bL^{-1/r} \rrparenthesis \subset \hat{\sM}'_{\kay,r}.
\end{equation}

\begin{example}
Rational double points (also known as Du Val singularities) are exactly the canonical surface singularities or, equivalently, the rational Gorenstein surface singularities. These can be further characterized as those surface singularities whose minimal resolution is crepant. As such, the dual graphs associated to their minimal resolution are the simply laced Dynkin diagrams. See \cite[Example 3.26]{KollarSingulaitieofMMP} for further details. Following \cite[\S 3]{ArtinCoveringsOfTheRtionalDoublePointsInCharacteristicp}, it is standard to denote the formal germ at a rational double point by $X_n^r$ where $X\in \{A,D,E\}$ and $n,r$ vary in between certain intervals of nonnegative integers. However, if $p\geq 7$, the index $r$ is superfluous and in that case we only have the so-called standard forms of the rational double points. Else; $p=2,3,5$, we obtain an interesting zoo of non-standard forms; \cite[Ibid.]{ArtinCoveringsOfTheRtionalDoublePointsInCharacteristicp} for details. In general, we let $X_n=X^0_n$ denote the standard forms, which can be realized as linear quotients $\hat{\bA}^2/G$ by a finite subgroup $G \subset \mathrm{SL}_2(\kay)$. As far as stringy motives are concerned, what matters is that the dual graph of the minimal resolution of $X_n^r$ is the simply laced Dynkin diagram $X_n$ with vertices of weight $2$. In particular, the vertices correspond to rational smooth curves intersecting transversely at exactly one point on the minimal resolution.  Let us denote by $M_{\mathrm{st}}(X_n^r)$ the stringy invariant $M_{\mathrm{st}}(X)_x$ associated to a rational double point $x \in X$. Since the minimal resolution $\tilde{X} \to (x, X)$ is crepant (i.e., $a_i=1$), $M_{\mathrm{st}}(X_n^r)$ equals the motive associated to the exceptional set of curves (which is a tree of $n$ $\mathbb{P}^1$'s intersecting as prescribed by the diagram $X_n$). An easy calculation then shows:
\[
M_{\mathrm{st}}(X_n^r) = n \bL + 1.
\]
Further, using \autoref{eqn.GlobalvsLocalSurfaceCase}, we see that:
\[
M_{\st}(\bA^2/G) = \big\{ \bA^2/G \setminus \{0\}  \big\} + M_{\st}(X_n) = \{\bA^2/G \} -1 + n \bL +1 = \bL^2+n\bL 
\]
for the standard forms. Here, we used $\{\bA^2/G \} = \bL ^2$, which follows from the relations imposed in the definition of $\hat{\sM}'_{\kay,r}$. 

It is well-known that there are many Galois quasi-\'etale covers among rational double points; see \cite[p. 15]{CarvajalMaPosltraSchwedeTickerCoversOfRationalDoublePoints}. For instance, if $p \neq 3$, there is a degree-$3$ Galois quasi-\'etale cover $D_4 \to E_6$. Let us consider the action of $\bZ/3$ on $D_4$. Recall that the dual graph of $D_4$ is the one in \autoref{fig:D_4}. Then, one readily sees that the action of $\bZ/3$ on this graph is given by fixing the node and permuting the other three vertices cyclically. The quotient graph is then displayed in \autoref{fig:quotientGraph}. One vertex represents a copy of $\mathbb{P}^1$ (which is the quotient of the disjoint union of three $\mathbb{P}^1$'s under the trivial cyclic action) whereas the other vertex represents the quotient of $\mathbb{P}^1$ under the action of $\bZ/3$ (where the points $0,1, \infty$ are being permuted cyclically). This let us conclude that
\[
M_{\mathrm{st}}(D_4)/(\bZ/3) = \bL +  \Bigl\{\bigl(\mathbb{P}^1\setminus \{0,1,\infty\} \bigr)\bigm/{(\bZ/3)}\Bigr\} + 1 = 2 \bL + 1.
\]
\begin{figure}[H]
\centering
\begin{minipage}{.5\textwidth}
  \centering
  \[
\xymatrix{\circ\ar@{-}[r]  & \circ\ar@{-}[l]\ar@{-}[r]\ar@{-}[d]  & \circ\ar@{-}[l]\\
   & \circ\ar@{-}[u] 
}
\]
  \captionof{figure}{Dual graph of $D_4$.}
  \label{fig:D_4}
\end{minipage}%
\begin{minipage}{.5\textwidth}
  \centering
  \[
\xymatrix{\circ\ar@{-}[r] & \circ }
\]
  \captionof{figure}{Quotient of the dual graph of $D_4$ by the action of $\bZ/3$.}
  \label{fig:quotientGraph}
\end{minipage}
\end{figure}
Similarly, if $p \neq 2$, there is an action of $\bZ/2$ on $E_6$ whose quasi-\'etale quotient is $E_7$. To compute $M_{\mathrm{st}}(E_6)/(\bZ/2)$, we note that the induced action of $\bZ/2$ on the dual graph of $E_6$---being it depicted in \autoref{fig:E_6}---is given by reflection across the symmetry axis. This implies that the quotient graph is the one in \autoref{fig:E_6Quotient}. 
\begin{figure}[H]
\centering
\begin{minipage}{.5\textwidth}
  \centering
  \[
\xymatrix{
\circ\ar@{-}[r] & \circ\ar@{-}[r]  & \circ\ar@{-}[l]\ar@{-}[r]\ar@{-}[d]  & \circ\ar@{-}[l] & \circ\ar@{-}[l]\\
 &  & \circ\ar@{-}[u]  &
}
\]
  \captionof{figure}{Dual graph of $E_6$.}
  \label{fig:E_6}
\end{minipage}%
\begin{minipage}{.5\textwidth}
  \centering
  \[
\xymatrix{\circ\ar@{-}[r] & \circ\ar@{-}[r] & \circ & \circ\ar@{-}[l]}
\]
  \captionof{figure}{Quotient of the dual graph of $E_6$ by the action of $\bZ/2$.}
  \label{fig:E_6Quotient}
\end{minipage}
\end{figure}

As before, we then have:
\[
M_{\mathrm{st}}(E_6)/(\bZ/2) = \bL +  \Bigl\{\bigl(\mathbb{P}^1\setminus \{0,1,\infty\} \bigr)\bigm/{(\bZ/2)}\Bigr\} + \bL-1 + \bL +3 = 4\bL+1.
\]
Finally, we consider the example $A_{2n-5} \xrightarrow{\bZ/2} D_n$ where $p \neq 2$ and $n\geq 4$. The dual graph is as in \autoref{fig:straight} with $2n-5$ vertices, which is an odd number. The only automorphism of such graph is the one switching the two branches (with $n-3$ vertices each) sticking out of the middle point which remains fixed under the action. Thus, the quotient graph is another straight graph but with $n-2$ vertices. By a similar computation to the ones above, we get:
\[
M_{\mathrm{st}}(A_{2n-5})/(\bZ/2) = (n-2) \bL+1.
\]
There is also the trivial example $A_0 \xrightarrow{G} X_n$ whenever $(p,\#G)=1$. Since $A_0$ is the smooth germ, $M_{\st}(A_0)=1$ and so $M_{\st}(A_0)/G = 1$ for all $G$.

Observe that we have built a tower (for $p \neq 2,3$):
\[
E_7 \xleftarrow{\bZ/2} E_6 \xleftarrow{\bZ/3} D_4 \xleftarrow{\bZ/2} A_3 \xleftarrow{\bZ/4} A_0,
\]
where the displayed arrows are all Galois quasi-\'etale covers whose Galois group is the one displayed on top of them. Note that the composite arrows $E_7 \xleftarrow{ S_3 } D_4$ and the ones $X_n \leftarrow A_0$ (i.e., those with $A_0$ as the source) are all Galois quasi-\'etale covers as well (with Galois group displayed on top). Here, $S_n$ denotes the $n$-th symmetric group. Nonetheless, the composite arrow $E_6 \leftarrow A_3$ is not Galois although it is quasi-\'etale.\footnote{Indeed, if it were Galois its Galois group would have to be cyclic as $S_3$ does not have $\bZ/2$ as a normal subgroup. In that case, $\bZ/3$ would have to act non-trivially on $A_3$ which is impossible as it cannot act non-trivially on its dual graph.} In particular,
\[
E_7 \xleftarrow{} E_6 \xleftarrow{} D_4 \leftarrow A_0
\]
is a tower of Galois quasi-\'etale covers inducing (by the computations we had done above) the following sequence on stringy motives
\[
M_{\mathrm{st}}(E_7)/1 = 7 \bL + 1 > M_{\mathrm{st}}(E_6)/(\bZ/2) = 4 \bL+1 > M_{\mathrm{st}}(D_4)/S_3 = 2 \bL+1 > M_{\mathrm{st}}(A_0)/\text{BO} = 1, 
\]
where $\text{BO}$ denotes the binary octahedral group. 
\end{example}
\begin{example} \label{ex.ConesOverRationalNormalCurves}
Let $X_d$ be the (affine) cone over the rational normal curve $C \subset \bP^d$ and $0\in X_d$ be its vertex. As we had seen in \autoref{ex.ConesOverFanos}, $M_{\st}(X_d)_0 = (\bL^2-1)/(\bL^{2/d}-1)$ and its minimal resolution consists of blowing up $0 \in X_d$ so that its dual graph is just one vertex. On the other hand, we may think of $X_d$ as a quotient singularity $\bA^2/\bm{\mu}_d$ as it is the spectrum of the $d$-th Veronese subring of $\kay[x,y]$. In fact, $\bA^2 \to X_d$ is a connected $\bm{\mu}_d$-torsor over $X_d \setminus \{0\}$. Thus, if $d=nm$, we may think of $X_d$ as a quotient $X_n/\bm{\mu}_m$. Assume now $p\nmid m$. By the triviality of the dual graphs in this example, $M_{\st}(X_n)_0/\bm{\mu}_m = M_{\st}(X_n)_0$. Furthermore,
\begin{align*}
M_{\st}(X_d)_0 = \frac{\bL^2-1}{\bL^{2/d}-1} = (\bL^2-1)\sum_{i=1}^{\infty} \bL^{\frac{-2i}{d}} &= (\bL^2-1)\sum_{i=1}^{\infty} \bL^{\frac{-2im}{d}} + (\bL^2-1)\sum_{(i,m)=1} \bL^{\frac{-2i}{d}}\\
&= M_{\st}(X_n)_0 +(\bL^2-1)\sum_{(i,m)=1} \bL^{\frac{-2i}{d}} > M_{\st}(X_n)_0.
\end{align*}
\end{example}

\section{Strict Descent in the Presence of Ramification} \label{sec.StrictDecentRamification}
In this section, we establish the core result of this work. That is, we explain why stringy motives get smaller across ramified Galois quasi-\'etale covers. The main step is to show that there is an abundance of non-liftable arcs. \subsection{Abundance of non-liftable arcs} Let us consider the following setup.
\begin{setup} \label{set:MainSetup}
Let $(X,\Delta_X)$ be a $d$-dimensional KLT log pair. Let $g \colon (Y,\Delta_Y) \to (X,\Delta_X)$ be a Galois quasi-\'etale log-cover with Galois group $G$. As customary, $g_{\infty}  \colon J_{\infty} Y \to J_{\infty} X$ is the induced map by functoriality. Further, let $\phi \colon (\tilde{X}, \Delta_{\tilde{X}}) \to (X,\Delta)$ be a log resolution; where we mean that $\phi$ is crepant and $\phi_* \Delta_{\tilde{X}} = \Delta_X$. Let $\psi \colon (\tilde{Y},\Delta_{\tilde{Y}}) \to (Y,\Delta_Y)$ be the normalization of $\phi$ along $g$ (i.e., with respect to $K(Y)$) and $\tilde{g} \colon (\tilde{Y}, \Delta_{\tilde{Y}}) \to (\tilde{X},\Delta_{\tilde{X}})$ be the induced morphism. Schematically:
\begin{equation} \label{eqn.CommSquareSetup}
\xymatrix{
(\tilde{X},\Delta_{\tilde{X}}) \ar[d]_-{\phi} & (\tilde{Y},\Delta_{\tilde{Y}}) \ar[d]^-{\psi} \ar[l]_-{\tilde{g}} \\
(X,\Delta_X) & (Y,\Delta_Y) \ar[l]_-{g}
}
\end{equation}
Note that $\Delta_{\tilde{Y}}$ is defined as the $\bQ$-divisor on $\tilde Y$ that makes both $\psi$ and so $\tilde{g}$ crepant and $\psi_* \Delta_{\tilde{Y}} = \Delta_Y$. Also, $\psi$ is a proper birational morphism whereas $\tilde{g}$ is a $G$-cover; i.e., a finite cover whose extension of function fields is Galois with Galois group $G$.
\end{setup}

\begin{remark}[On the nature of \'etaleness for $G$-covers]
It is well-known to experts that a finite dominant morphism $g\colon Y \to X$ between smooth varieties is \'etale if and only if it is unramified. In fact, for this to work, we only need $X$ and $Y$ to be; respectively, regular and Cohen--Macaulay. Indeed, one uses \cite[Corollary 18.17]{EisenbudCommutativeAlgebraWithAView} to conclude that finiteness implies flatness in that case. In particular, if $g$ is further a $G$-cover, then $g$ is \'etale if and only if for all $\kay$-points $x\in X$ the set-theoretic fiber $Y_x(\kay)$ is a $G$-torsor under the induced action. Nevertheless, the same principle works with no hypotheses whatsoever on the singularities of $Y$ and $X$ except for normality. Indeed:
\begin{claim} \label{claim:NatureOfEtaleNessGCovers}
With notation as above, suppose that $g$ is a $G$-cover between normal varieties. If the set-theoretic fiber $Y_x(\kay)$ is a $G$-torsor for all $x \in X(\kay)$ then $g$ is \'etale.
\end{claim}
\begin{proof}[Proof of claim]
We need to prove that $g_* \sO_Y$ is locally free and that $g$ is unramified, which can be checked at every point $x \in X(\kay)$. Set $\sO_{Y,x} \coloneqq (g_* \sO_Y)_x$. Observe that the generic rank of $\sO_{Y,x}/\sO_{X,x}$ is $\big[K(Y)/K(X)\big] = \# G$. By \cite[II, Lemma 8.9]{Hartshorne}, the local freeness claim would follow if we prove this to be its residual rank as well. To do so, let us pullback $\sO_{Y,x}/\sO_{X,x}$ to the completion of $\sO_{X,x}$ to get $\hat{\sO}_{X,x} \to \hat{\sO}_{Y,x} = \prod_{y \in Y_x(\kay)} \hat{\sO}_{Y,y}$, where $\# Y_x(\kay) = \# G$ by hypothesis. Now, the generic rank of $\hat{\sO}_{Y,x}/\hat{\sO}_{X,x}$ is still $\# G$ as ranks are invariant under completion.\footnote{Indeed, let $M$ be a finitely generated module over a local excellent normal domain $(R,\fram)$ of generic rank $r$. Write a short exact sequence $0 \to R^{\oplus r} \to M \to T \to 0$ where $T$ is torsion; such a sequence is obtained by choosing a basis of $M_{(0)}$ over $R_{(0)}$ consisting of elements in the image of $M \to M_{(0)}$. Base changing the sequence by the flat extension $\hat{R}/R$ and noting that $T \otimes_R \hat{R}$ is torsion gives the desired result. Note that $\hat{R}$ is necessarily a domain \cite[\href{https://stacks.math.columbia.edu/tag/0C23}{Tag 0C23}]{stacks-project}.} Therefore, $\hat{\sO}_{Y,y}/\hat{\sO}_{X,x}$ must have generic rank $1$. Since these rings are assumed normal, we have $\hat{\sO}_{Y,y}\simeq \hat{\sO}_{X,x}$; as required. Along the way, we proved that $\hat{\sO}_{Y,y}/\hat{\sO}_{X,x}$ and so $\sO_{Y,y}/\sO_{X,x}$ are unramified.
\end{proof}
We use this freely in what follows. This finishes our remarks.
\end{remark}

In proving that there are abundant non-liftable arcs of $X$ across $g\colon Y \to X$, the following two lemmas are crucial.

\begin{lemma}[Reduction to the log smooth case] \label{pro.FirstKeyproposition}
Work in \autoref{set:MainSetup} and assume $d\leq 3$. In case $d = 3$ and $0< p \leq 5$, assume that $G$ is a $p$-group. If $\tilde{g}$ is \'etale then so is $g$. Equivalently, if $g$ is not \'etale then neither is $\tilde{g}$. 
\end{lemma}
\begin{proof}[Elementary proof of \autoref{pro.FirstKeyproposition} for surfaces]
We present first a proof in the surfaces case and discuss its extension to threefolds in \autoref{sec.ProofOfReductionToLogSmooth} below. We may assume that $X$ is affine with a $\kay$-rational isolated singularity $x\in X$; so $Y$ is affine too. Since $X$ has rational singularities, $\Exc{\phi}$ is a tree of $\bP^1$'s (see \cite[Theorem 4.1]{EsnaultViehwegSurfaceSingularitiesDominatedSmoothVarieties} for a characterization of when this is the case). In particular, $\Exc \phi$ is algebraically simply connected which implies that the restriction of $\tilde{g}$ to $\Exc \phi$ must consist of the disjoint union of $\# G$ copies of $\Exc \phi$.\footnote{Alternatively, one may use that $\Exc \phi$ is a tree of $\bP^1$'s to show that $\tilde{g}$ induces a local isomorphism over the corresponding dual graph, which is contractible and so simply connected.} However, such a disjoint union is the exceptional locus of $\psi$. This means that over $x$ there must lie $\# G$ many points in $Y$. That is, $\# G = \#Y_x(\kay)$ and so $g$ is \'etale.
\end{proof}
The following lemma and corollary should be thought of as a strengthening of results of Kato and Kerz--Schmidt; see \cite[Lemma 3.5]{KatoWildRamificationRestrictionCurves} which is based on \cite[Lemma 2.4]{KerzSchmidtOnDifferentNotionsOfTameness}. Our proofs  of these results were inspired by arguments in the recent paper of Nakamura--Shibata \cite{NakamuraShibataInversionOdAdjunctionForQuotientSingularities}. 
In the proof of this lemma, we use the \emph{untwisting technique}, which we briefly recall next. For details of this technique, we refer the reader to \cite[\S 4 and \S 7]{YasudaWilderMcKayCorrespondence} or to \cite[\S 6]{YasudaMotivicIntegrationDeligneMumfordStack} for the stacky formulation. 

Consider a $G$-cover $g\colon Y\to X$ between normal varieties and let $g_\bD \colon Y_\bD \to X_\bD$ be its base change to $\bD$. By an arc of $X_\bD$, we mean a section $\bD \to X_\bD$. We denote by $J_\infty X_\bD$ the space of arcs of $X_\bD$. We have the obvious identification $J_\infty X=J_\infty X_\bD$. To each $G$-cover $\bE\to \bD$, we can associate a separated and flat $\bD$-scheme $Y^{|\bE|}_\bD$ of finite type called the \emph{untwisting scheme} \cite[Definition 7.2]{YasudaWilderMcKayCorrespondence}. It comes with a natural $\bD$-morphism $g_\bD^{|\bE|}\colon Y^{|\bE|}_\bD \to X_\bD$. The morphisms at the generic fiber
\[
g_{\bD,\eta} \colon Y_{\bD,\eta} \to X_{\bD,\eta}   \text{   and   } 
g_{\bD,\eta}^{|\bE|} \colon Y_{\bD,\eta}^{|\bE|}\to X_{\bD,\eta}
\]
are twisted forms of each other. That is, letting $\bar{\eta}$ denote the geometric generic point of $\bD$, there exists an isomorphism $Y_{\bD,\bar{\eta}} \to Y_{\bD,\bar{\eta}}^{|\bE|}$ over $X_{\bD,\bar{\eta}}$. It follows that these two morphisms have the same branch locus in $X_{\bD,\eta}$. It also follows that if $Y/\kay$ is smooth then $Y_{\bD,\eta}^{|\bE|}$ is smooth over $\bD^*$. The key property of the untwisting scheme is that there is a one-to-one correspondence
\[
 J_\infty^{\bE,G}Y_\bD \xleftrightarrow{1:1} J_\infty Y_\bD ^{|\bE|}\Bigm/\Aut(\bE)
\]
which is compatible with the maps to $J_\infty X_\bD$; see \cite[the last display of p.~140]{YasudaWilderMcKayCorrespondence}. 

Let $(J_\infty^{\bE,G}Y_\bD)^\natural\subset J_\infty^{\bE,G}Y_\bD$ be the locus of those $\bE$-twisted $G$-arcs that map some (and hence every) generic point into the \'etale locus of $g_{\bD}\colon Y_\bD \to X_\bD$ and let $\big(J_\infty Y_\bD^{|\bE|}\big)^\natural\subset J_\infty Y_\bD^{|\bE|}$ and $\big(J_\infty X_\bD\big)^\natural\subset J_\infty X_\bD$ be the loci of arcs of $Y_\bD^{|\bE|}$ and $X_\bD$  mapping $\eta$ into the \'etale loci of $g_{\bD}^{|\bE|} \colon Y_\bD^{|\bE|}\to X_\bD$, respectively. Note that all of these loci have a complement of measure zero. The above correspondence restricts to
\[
\big(J_\infty^{\bE,G}Y_\bD \big)^\natural \xleftrightarrow{1:1} \Big(J_\infty Y_\bD^{|\bE|}\Big)^\natural\Bigm/\Aut(\bE),\] 
since these are preimages of $\big(J_\infty X_\bD\big)^\natural\subset J_\infty X_\bD$. 
Maps 
\[
\big(J_\infty^{\bE,G}Y_\bD \big)^\natural \to \big(J_\infty X_\bD\big)^\natural
\text{ and }
\Big(J_\infty Y_\bD^{|\bE|}\Big)^\natural\Bigm/\Aut(\bE) \to \big(J_\infty X_\bD\big)^\natural
\]
are injective (\cf\ \cite[Proposition 6.16]{YasudaPAdic}).
 
\begin{remark}
By \cite[\S6.3 and Remark 14.5]{YasudaMotivicIntegrationDeligneMumfordStack}, the untwisting scheme $Y_\bD^{|\bE|}$ is the irreducible component (with the reduced structure) of the $G$-fixed locus of the Hom scheme $\underline{\mathrm{Hom}}_\bD(\bE,Y_\bD)$ (i.e., the Hom scheme parametrizing $G$-equivariant morphisms $\bE\to Y_\bD$) that dominates $\bD$.
In the stacky language, this corresponds to the irreducible component of the Hom stack $\underline{\mathrm{Hom}}^{\mathrm{rep}}_\bD(\sE,[Y_\bD /G])$ of representable morphisms. The generic fiber of the latter Hom stack, which is defined over $\bD$,  is \[
\underline{\mathrm{Hom}}^{\mathrm{rep}}_{\bD^*}(\bD^*,[Y_\bD /G]_\eta)
=[Y_\bD /G]_\eta. 
\]
See \cite[Remark 6.23]{YasudaPcyclicMackayCorrespondenceMotivicIntegration}. Thus, in the stacky setting, untwisting leaves the generic fiber unchanged. 
\end{remark}

\begin{lemma} \label{lem.PositiveMeasureLemma}
Let $g\colon Y\to X$ be a $G$-cover between normal varieties. Suppose that $Y$ is smooth and that $J_\infty^{\bE,G}Y_\bD$ is nonempty. Then, the subset
\[
g_{\bD,\infty} \Big(\big(J_\infty^{\bE,G}Y_\bD\big)^\natural\Big)=g_{\bD,\infty}^{|\bE|} \Big(\big(J_\infty Y_\bD^{|\bE|}\big)^\natural\Big)
\subset J_\infty X_\bD = J_\infty X
\]
is a measurable subset of positive measure.
\end{lemma}

\begin{proof}
Since the generic fiber $Y_{\bD,\eta}^{|\bE|}$ is smooth, there exists a N\'eron smoothening $h\colon W\to Y_\bD^{|\bE|}$; see \cite[\S3.1]{BoschLutkebohmertRaynaudNeronModels}. We have the following properties: $W$ is smooth over $\bD$, the induced morphism $W_\eta \to Y_{\bD,\eta}^{|\bE|}$ of generic fibers is an isomorphism, and $h_\infty \colon J_\infty W \to J_{\infty} Y_\bD^{|\bE|}$ is bijective.
Since $J_{\infty} Y_\bD^{|\bE|}$ is non-empty, then so is $J_{\infty} W$. Since $W$ is smooth, the total space $J_{\infty} W$ is a cylinder of  measure $\{W\times _\bD \kay \}\ne 0$. 
Let $\mathfrak{j}_h$ be the Jacobian order function of $h$. This function is bounded above as $h$ is an isomorphism on generic fibers. Thus, $J_\infty W =\bigsqcup_{i=0}^n \mathfrak{j}_h^{-1}(i)$ for all $n\gg 0$. For some $i$, $\mu_W\big(\mathfrak{j}_h^{-1}(i)\big)>0$. By the change of variables formula \cite[Theorem 8.0.5]{SebagIntegrationMotiviqueSurLesSchemasFormels}, it follows that
\[
\mu_{Y_\bD^{|\bE|}}\Big(h_\infty\big(\mathfrak{j}_h^{-1}(i)\big)\Big)=\bL^{-i}\mu_W\big(\mathfrak{j}_h^{-1}(i)\big)>0
\]
for any such $i$. This shows that  
\[
\mu_{Y_\bD^{|\bE|}}\Big(\big(J_\infty Y_\bD^{|\bE|}\big)^\natural\Big)=\mu_{Y_\bD^{|\bE|}}\big(J_\infty Y_\bD^{|\bE|}\big)>0. 
\]
To deduce the desired positivity from this, we apply a similar argument as above to $g_\bD^{|\bE|}\colon Y_\bD^{|\bE|}\to X_\bD$. However, we need a result on motivic integration in the equivariant situation. In characteristic zero, such a theory was developed in \cite{DenefLoeserMotivicIntegrationQuotientSingularitiesMcKayCorrespondnence}. We will use the stacky formulation developed in \cite{YasudaMotivicIntegrationDeligneMumfordStack} as an available theory having the necessary generality. 

Set $\sY \coloneqq [Y_\bD^{|\bE|}/\Aut(\bE)]$. 
The morphism $g_\bE^{|\bD|} \colon Y_{\bD}^{|\bE|} \to X_{\bD}$ factors through $\sY$ and hence the map $g_{\bD,\infty}^{|\bE|}$ factors as
\[
 J_\infty Y_\bD^{|\bE|} \to J_\infty Y_\bD^{|\bE|}\Bigm/\Aut(\bE)= J_\infty \sY \to J_\infty X_\bD . 
\]
The right map is injective outside a subset of measure zero.
Further, we have
\[
\mu_\sY(J_\infty \sY)=\mu_{Y_\bD^{|\bE|}}\big(J_\infty Y_\bD^{|\bE|}\big)\Bigm/\Aut(\bE)>0. 
\]
Let $\mathfrak{j}$ be the Jacobian order function of $\sY\to X_\bD$. We have the partition 
\[
J_\infty \sY =\bigsqcup_{i\in \bN \cup\{\infty\}} \mathfrak{j}^{-1}(i)
\]
by countably many measurable subsets. Since $\mathfrak{j}^{-1}(\infty)$ has measure zero, there exists $i\ne \infty$ such that $\mathfrak{j}^{-1}(i)$ has positive measure. For this $i$, from the change of variables (\cite[Theorem 10.26 or 11.13]{YasudaMotivicIntegrationDeligneMumfordStack}), we have
\[
 \mu_{X_\bD}\Big(g_{\bD, \infty}^{|\bE|}\big(\mathfrak{j}^{-1}(i)\big)\Big)=\bL^{-i}\mu_\sY\big(\mathfrak{j}^{-1}(i)\big)>0.
\]
The subset of the statement has the same measure as $g_{\bD,\infty} (J_\infty \sY)$, which contains $g_{\bD,\infty}^{|\bE|}\big(\mathfrak{j}^{-1}(i)\big)$. This proves the lemma. 
\end{proof}

\begin{corollary} \label{Cor.AbundanceNonLiftableArcs}
Let $G$ be a finite group and $H\subset G$ be a subgroup. Let $g\colon Y\to X$ be a $G$-cover such that $Y$ is smooth and there is $y \in Y(\kay)$ being fixed by $H$. 
Let  $\bE\to \bD$ be another $G$-cover such that $H$ fixes a connected component $\bE_0$ of $\bE$ (i.e. $\bE_0 \to \bE$ is $H$-equivariant). Let $\sN \subset J_{\infty} X$ be the subset of arcs $\gamma \colon \bD\to X$ such that: $\gamma(\delta) = g(y)$,  $g$ is \'etale over $\gamma(\eta)$, and the pullback of $\gamma$ along $g$ induces the given $G$-cover $\bE\to \bD$. Then, $\sN$ has positive measure.
\end{corollary}

\begin{proof}
First of all, note that $\sN$ is none other than $g_{\bD,\infty}\big((J_\infty^{\bE,G}Y_\bD)^\natural\big)$. Thus, in order to apply \autoref{lem.PositiveMeasureLemma}, it suffices to show that $J_\infty^{\bE,G}Y\ne \emptyset$; which we do next.  The $\bD$-scheme $Y_\bD$ has the $H$-invariant section
\[
\bD \xrightarrow{\simeq} \{y\}\times \bD \subset Y_\bD.
\]
The composition $\bE_0\to \bD \to Y_\bD$ is clearly an $H$-equivariant $\bD$-morphism. There exists a unique extension of this morphism to a $G$-equivariant morphism $\bE\to Y_\bD$, i.e., an $\bE$-twisted $G$-arc (a corresponding arc $\bD\to Y_\bD^{|\bE|}$, which is unique up to the $\Aut(\bE)$-action, corresponds to the ``trivial arc'' by Nakamura--Shibata; see \autoref{rem.NakamuraShibata'sWork} below.)
\end{proof}

\begin{remark}\label{rem.NakamuraShibata'sWork}
Nakamura--Shibata work over $\bA^1$, following the setting in \cite{DenefLoeserMotivicIntegrationQuotientSingularitiesMcKayCorrespondnence}. In \cite[Claim 5.2]{NakamuraShibataInversionOdAdjunctionForQuotientSingularities}, they use a result of Hacon--McKernan \cite{HaconMcKernanOnShokurovsRCC} about rational chain connectedness, which is based on a result by  Graber--Harris--Starr \cite{GraberHArrisStarrFamiliesofRCCVarieties}, to show that the ``trivial arc'' deforms. We work instead over the formal disk $\bD$ and use the N\'eron smoothening to deform the ``trivial arc.''
\end{remark}

\begin{proposition}[Abundance of non-liftable arcs] \label{pro.SecondProposition} 
Work in \autoref{set:MainSetup} with $d \leq 3$. If $d = 3$ and $0<p \leq 5$, further assume that $G$ is a $p$-group. If $g$ is not \'etale then there is an arc $\gamma \colon \bD \to X$ that does not lift to $Y$ and so $g_{\infty} \colon J_{\infty} Y \to J_{\infty} X$ is not surjective. Moreover, $J_{\infty} X \setminus g_{\infty} ( J_{\infty} Y )$ has positive measure. Similarly, if $G$ fixes $y \in Y(\kay)$ and $x=g(y)$, then $(J_{\infty} X)_x \setminus g_{\infty} ( (J_{\infty} Y)_y )$ has positive measure.
\end{proposition}

\begin{proof}
Since $g$ is not \'etale, neither is $\tilde{g} \colon \tilde{Y} \to \tilde{X}$ by \autoref{pro.FirstKeyproposition}. The map $\phi_\infty \colon J_\infty \tilde{X} \to J_\infty X$ is almost bijective and the subsets $g_\infty (J_\infty Y) $ and $\tilde{g}_\infty (J_\infty \tilde{Y}) $ correspond to each other  through this almost bijection (outside subsets of measure zero). Moreover, the map sends a subset of positive measure of $J_\infty \tilde{X}$ to such of $J_\infty X$. Therefore, it suffices to show the proposition for $\tilde{g}$ in place of $g$. Now, there exists a point $y\in \tilde{Y}(\kay)$ fixed by a nontrivial \emph{cyclic} subgroup $1\ne H\subset G$ of prime order. In particular, there exists a connected $H$-cover $\bE_0\to \bD$. Indeed, if the order of $H$ is a prime $\ell \neq p$, we can then take the cover $\Spec \kay \llbracket t^{1/\ell}\rrbracket \to \Spec \kay \llbracket t\rrbracket$. If the order is $p$, we can take, for example, the Artin--Schreier extension $\kay\llparenthesis t \rrparenthesis[u]/ (u^p -u -t^{-1})$ of $\kay\llparenthesis t \rrparenthesis$ and the corresponding cover of $\bD$. We can extend $\bE_0\to \bD$ to a $G$-cover $\bE\to \bD$. We now just apply \autoref{Cor.AbundanceNonLiftableArcs} to the above $y$ and $\bE$ to obtain the result. 
\end{proof}

\begin{remark} \label{rem.StrongerStatement}
In proving \autoref{pro.SecondProposition}, we did not use \autoref{Cor.AbundanceNonLiftableArcs} in its full power. The authors believe a sharper analysis may give a stronger version of \autoref{pro.SecondProposition} which may help in showing the DCC property by providing a stronger descent. See \autoref{que.DCCForThreefolds} below.
\end{remark}

\subsection{On the proof of \autoref{pro.FirstKeyproposition}} \label{sec.ProofOfReductionToLogSmooth}

We come back now to the proof of \autoref{pro.FirstKeyproposition}. We concentrate now on the case of threefolds. The substantial difference is that in the $2$-dimensional case we can rely on rationality whereas in the $3$-dimensional case only on $W\sO$-rationality, which is much less elementary than rationality. Recall that KLT threefold singularities are $W\sO$-rational; see \cite{HaconWitaszekMMPLowCharacteristic,GongyoNakamuraTanakaRationalPointsFanoThreefoldsFiniteFields}. However, these are rational in characteristics $p\geq 7$; see \cite{ArvidssonBernasconiLAciniKVVDelPezzop>5,HaconWitaszekRationalityKawamataLogTerminalPosChar}. It is then far from obvious how what we wrote above in the proof of \autoref{pro.FirstKeyproposition} can be generalized to threefolds and beyond. 

The key idea, however, was to exploit the topological simplicity of $\Exc \phi$ granted by the rationality of $X$. Due to the low dimensions involved, the theorem of formal functions can be utilized to prove that $H^1(\tilde{X}_x, \sO_{\tilde{X}_x}) = 0$ from the vanishing $\sO_X \xrightarrow{\simeq}\mathbf{R} \phi_* \sO_{\tilde{X}}$ (where ``$\simeq$'' means a quasi-isomorphism in the appropriate derived category); see \cite{ArtinOnIsolatedRationalSingularities}. Moreover, one may ``thicken'' $\tilde{X}_x$ if necessary to ensure that $H^0(\tilde{X}_x, \sO_{\tilde{X}_x}) = \kay$. More precisely, there is a minimal effective divisor $D=\sum_i n_i E_i$ on $\tilde{X}$ such that the $\{E_i\}$ are the irreducible components of $\tilde{X}_x$, $n_i \geq 1$ for all $i$, and $\sO_{\tilde{X}}(-D)$ is relatively nef over an open neighborhood of $x$ (i.e. $D\cdot E_i \leq 0$ for all $i$). It then follows that $\chi(\tilde{X}_x, \bar{\sO}_{\tilde{X}_x})=1$ where
\[
\bar{\sO}_{\tilde{X}_x} \coloneqq \sO_{\tilde{X}}/\sO_{\tilde{X}}(-D).
\]
See \cite{ArtinOnIsolatedRationalSingularities} for details; where Artin referred to $D$ as the \emph{fundamental cycle}.\footnote{We may always replace $D$ for any larger divisor and obtain the same Euler characteristic; the fundamental cycle is just minimal with that property.} This can be further exploited to prove that $\Exc \phi = (\tilde{X}_x)_{\mathrm{red}}$ is a tree of $\bP^1$'s; see \cite{LipmanRationalSingularities,EsnaultViehwegSurfaceSingularitiesDominatedSmoothVarieties}, which is what we had used above. 

We start by reproving \autoref{pro.FirstKeyproposition} only knowing that $\chi(\tilde{X}_x, \bar{\sO}_{\tilde{X}_x}) = 1$. 

\begin{proof}[Proof of \autoref{pro.FirstKeyproposition}: Surface case]
Since $\tilde{g}$ is \'etale, $\psi$ is a resolution of singularities. In particular, $\sO_Y \xrightarrow{\simeq} \mathbf{R} \psi_* \sO_{\tilde{Y}}$ as $Y$ also has rational singularities. Now, let us consider the fiber of \autoref{eqn.CommSquareSetup} at a point $x \in X(\kay)$:
\begin{equation} \label{eqn.FiberDiagram}
\xymatrix{\tilde{X}_x \ar[d]_-{\phi_x} & \tilde{Y}_x \ar[d]^-{\psi_x} \ar[l]_-{\tilde{g}_x} \\x & Y_x 
\ar[l]_-{g_x}}
\end{equation}
Let $D$ be the fundamental cycle of $x$ and define $\bar{\sO}_{\tilde{X}_x}$ as before. By the Hirzebruch--Riemann--Roch theorem: 
\[
  \chi(\tilde{Y}_x,\tilde{g}_x^* \bar{\sO}_{\tilde{X}_x}) = \# G \cdot \chi(\tilde{X}_x, \bar{\sO}_{\tilde{X}_x}) = \# G \] as $\tilde{g}_x$ is finite \'etale of degree $\# G$; see \cite[Example 18.3.9]{FultonIntersection}. However, since $g$ is \'etale, it also follows that
\[
\tilde{g}_x^* \bar{\sO}_{\tilde{X}_x} = \sO_{\tilde{Y}}/\sO_{\tilde{Y}}(-g^* D),
\]
where $-g^* D$ is relatively nef over a neighborhood of $x \in X$. In particular, $-g^* D$ is relatively nef over an open neighborhood of each $y \in Y(\kay)$ lying over $x$. Therefore, arguing à la Artin, we conclude that 
\[
 \chi(\tilde{Y}_x,\tilde{g}_x^* \bar{\sO}_{\tilde{X}_x}) = \dim_{\kay} H^0(\tilde{Y}_x,\tilde{g}_x^* \bar{\sO}_{\tilde{X}_x}) = \# g^{-1}(x) \leq \dim_{\kay} \sO_{Y_x}.
\]
Since $H^0(\tilde{Y}_x,\tilde{g}_x^* \bar{\sO}_{\tilde{X}_x})$ is an $\sO_{Y_x}$-algebra, this let us conclude that it must be trivial and so $\# G = \dim_{\kay} \sO_{Y_x}$. Thus, $g_* \sO_Y$ is locally free; say by \cite[II, Exercise 5.8.(c)]{Hartshorne}. Since $g$ is further quasi-\'etale, we conclude that $g$ is \'etale by the purity of the branch locus for faithfully flat finite covers \cite[VI, Theorem 6.8]{AltmanKleimanIntroToGrothendieckDuality}.
\end{proof}

Next, we remark that the above proof can be carried out without passing to the fibers by means of the Grothendieck--Riemann--Roch theorem (GRR). This has the salient feature of letting us extend our proof to higher dimensions yet still assuming rationality.

\begin{proof}[Proof of \autoref{pro.FirstKeyproposition}: General rational case]
Here, we may assume that $d$ is arbitrary but that $X$ and $Y$ have rational singularities (e.g. KLT surface singularities or KLT threefold singularities in characteristic $p \neq 2,3,5$ \cite{ArvidssonBernasconiLAciniKVVDelPezzop>5,HaconWitaszekRationalityKawamataLogTerminalPosChar}). As before, we consider the proper morphism $\gamma \colon \tilde{Y} \to X$ where $\phi \circ \tilde{g}
=\gamma = g \circ  \psi$. We use the rationality of $X$ to say that $\phi_{!} \sO_{\tilde X} = \sO_X$. Similarly, since $\tilde{g}$ is \'etale and so $\psi$ is a resolution, we may say $\psi_! \sO_{\tilde{Y}} = \sO_Y$. Now, being $g$ and $\tilde{g}$ affine morphisms, we have $g_{!} \sO_Y = g_* \sO_Y$ and likewise for $\tilde{g}$. 

The point now is that we may compute $\chern(\gamma_! \sO_{\tilde{Y}})$ in two different ways. First:
\[
\chern(\gamma_! \sO_{\tilde{Y}}) = \chern(\phi_* \tilde{g}_* \sO_{\tilde{Y}}) = \phi_*\big( \chern(\tilde{g}_* \sO_{\tilde{Y}}) \cdot \td(\sT_{\phi})\big) = \phi_*\big( \tilde{g}_*(\td(\sT_{\tilde{g}})) \cdot \td(\sT_{\phi})\big),
\]
where we used GRR in the last two equalities. Next, we use that $\tilde{g}$ is a finite \'etale map of degree $\# G$ to say $\tilde{g}_*(\td(\sT_{\tilde{g}})) = \# G$ as $\sT_{\tilde{g}} = 0$. This let us conclude that:
\[
\chern(\gamma_! \sO_{\tilde{Y}}) = \# G \cdot \phi_* \big(\td(\sT_{\phi})\big) = \#G \cdot \chern(\phi_! \sO_{\tilde{X}}) =\#G \cdot \chern(\sO_X) = \#G,
\]
where the second equality uses GRR. On the other hand, we may compute $\chern(\gamma_! \sO_{\tilde{Y}})$ as follows:
\[
\chern(\gamma_! \sO_{\tilde{Y}}) = \chern(g_* \psi_* \sO_{\tilde{Y}}) = \chern(g_* \sO_Y).
\]
In this manner, we conclude:
\[
\chern(g_* \sO_Y) = \# G.
\]
Next, we recall that Chern characters commute with pullbacks. Then,
\[
\# G = \chern\big(g_* \sO_Y \otimes_{\sO_X} \kappa(x)\big) = \dim_{\kappa(x)} \big(g_* \sO_Y \otimes_{\sO_X} \kappa(x) \big)
\]
for all $x \in X$. Using \cite[II, Exercise 5.8.(c)]{Hartshorne} gives that $g_* \sO_Y$ is locally free, i.e., $g$ is faithfully flat. We conclude as before that $g$ is \'etale by purity \cite[VI, Theorem 6.8]{AltmanKleimanIntroToGrothendieckDuality}.
\end{proof}

As mentioned in \autoref{sec.Introduction}, in studying \autoref{que.questionIntro.} in dimensions $\leq 3$, the most important case is the purely wild one (i.e., when all the Galois groups are $p$-groups) as it is the one left unsolved by the minimal model program. Under such hypothesis, it turns out that one can give simpler proofs for \autoref{pro.FirstKeyproposition}. Let us start with the rational case.

\begin{proof}[Proof of \autoref{pro.FirstKeyproposition}: Rational and purely wild case]
Say $\#G = p^e$ and both $X$ and $Y$ have rational singularities of any dimension. There is short exact sequence $0 \to \bZ/p \to G \to H \to 1$ where $\# H = p^{e-1}$. Thus, by induction on $e$, we may assume $G=\bZ/p$. Consider the following diagram induced from the Artin--Schreier theory:
\[
\xymatrix{
H^0(\tilde X, \mathcal{O}_{\tilde X}) \ar[r] & H^1_{\mathrm{\acute{e}t}}(\tilde X, \bZ/p) \ar[r] & H^1(\tilde X, \mathcal{O}_{\tilde{X}}) \\
H^0( X, \mathcal{O}_X) \ar[r]\ar@{=}[u] & H^1_{\mathrm{\acute{e}t}}( X, \bZ/p) \ar[r]\ar[u] & H^1( X, \mathcal{O}_X)\ar[u]
}
\]
Since our problem is local, we may assume that $X$ is affine and so $H^1(X,\sO_X)=0$. Since $X$ has rational singularities, then $H^1(\tilde X, \mathcal{O}_{\tilde{X}}) = 0$. Therefore, the cover $\tilde{g}\colon\tilde Y \to \tilde X$; which can be regarded as an element of the upper middle group $H^1_{\mathrm{\acute{e}t}}(\tilde X, \bZ/p)$, is realized as the normalized pullback of a $\bZ/p$-torsor over $X$. Namely, the original cover $g\colon Y\to X$ is \'etale.
\end{proof}

There is the notion of $W\sO$-rationality, which is weaker than rationality. There are KLT singularities which are not rational but $W\sO$-rational. Therefore, it is natural to look for generalization of above arguments using $W\sO$-cohomology. 
We invite the reader to consult \cite[\S 3]{PatakfalviZdanowiczOrdinaryVarietiesTrivialCanoniucalNOtUniruled} for an excellent account on Witt vector and $p$-adic \'etale cohomology as well as $W\sO$-rationality. Also see \cite{GongyoNakamuraTanakaRationalPointsFanoThreefoldsFiniteFields,ChatzistamatiouRullingHodgeWittCohomologyWittRationalSingularities,BerthelotBlochEsnaultWittVectorCohomologySingularVarieties,Chambert-LoirPointsRationnels,LeStumRigidCohomology}. This approach works well at least in the case where the Galois group is a $p$-group, as in the proof below.

\begin{proof}[Proof of \autoref{pro.FirstKeyproposition}: $W\sO$-rational and purely wild case]
Suppose that $\#G$ is a $p$-group, both $X$ and $Y$ have $W\sO$-rational singularities (e.g. KLT singularities in dimension $\leq 3$ \cite{HaconWitaszekMMPLowCharacteristic,GongyoNakamuraTanakaRationalPointsFanoThreefoldsFiniteFields}), and their dimension $d$ is arbitrary. In particular, $\phi$ is $W\sO$-rational: $W\sO_{X,\bQ} \xrightarrow{\simeq} \phi_*W\sO_{\tilde{X},\bQ}$ and $R^i\phi_*W\sO_{\tilde{X},\bQ} = 0$ for all $i>0$. According to \cite[Lemma 3.19]{PatakfalviZdanowiczOrdinaryVarietiesTrivialCanoniucalNOtUniruled}, $\phi$ is further $\bQ_p$-rational which means that $\phi_*\bQ_p=\bQ_p$ and $R^i \phi_* \bQ_p=0$ for all $i>0$. 

Now, let $x \in X(\kay)$ be arbitrary and consider the fiber diagram \autoref{eqn.FiberDiagram}; where $\tilde{X}_x$ is a proper connected $\kay$-scheme. Using proper base change for $p$-adic \'etale cohomology \cite[Proposition 3.16]{PatakfalviZdanowiczOrdinaryVarietiesTrivialCanoniucalNOtUniruled} (c.f. \cite[Expos\'e VI, 2.2.3 B]{SGA5}), we conclude that $H^0_\mathrm{\acute{e}t}(\tilde{X}_x, \bQ_p)= \bQ_p$ and $H^i_{\mathrm{\acute{e}t}}(\tilde{X}_x, \bQ_p)= 0$ if $i>0$. In particular, the $p$-adic Euler--Poincar\'e characteristic of the fiber is $\chi_{\mathrm{\acute{e}t}}(\tilde{X}_x, \bQ_p) = 1$. Now, since $\tilde{g}_x$ is a $G$-torsor with $G$ a $p$-group, we may use Crew's formula; see \cite{CrewPCoversInCharacteristicP}, \cite[Lemma 8.5]{Chambert-LoirPointsRationnels}, to get that $\chi_{\mathrm{\acute{e}t}}(\tilde{Y}_x,\bQ_p) = \# G \cdot \chi_{\mathrm{\acute{e}t}}(\tilde{X}_x,\bQ_p) = \# G$.

Next, since $\tilde{g}$ is \'etale, $\psi$ is a resolution and so it is $\bQ_p$-rational. Hence, by proper base change for $p$-adic \'etale cohomology once again, we get that $\chi_{\mathrm{\acute{e}t}}(\tilde{Y}_x,\bQ_p)= H^0(Y_x,\bQ_p) = \#Y_x(\kay)$. 

Putting these two observations together, $\# G = \#Y_x(\kay)$ for all $x \in X(\kay)$. We finish by using \autoref{claim:NatureOfEtaleNessGCovers}.
\end{proof}

We may wonder whether $W\sO$-rationality and/or rationality are necessary for the statement of \autoref{pro.FirstKeyproposition} (and of \autoref{pro.SecondProposition}). The following example shows this to be the case.

\begin{example}[{\cf \cite[Remark 5.3]{NakamuraShibataInversionOdAdjunctionForQuotientSingularities}}]
Let $E$ be an ordinary elliptic curve and $O\ne P\in E$ be a $p$-torsion point. The automorphism of $E$ sending a point $Q$ to $Q+P$ defines an action of $\bZ/p $ on $E$. This action lifts to one on the ample invertible sheaf $\mathcal O_E (P_0+\cdots +P_{p-1})$ with $P_i$ the $p$-torsion points. We also have natural actions on the affine cone $C$ and the standard resolution $\tilde C$ over $E$ associated to this invertible sheaf; the latter is a line bundle over $E$. The action on $C$ has the unique fixed point; namely the vertex of the cone, whereas the action on $\tilde C$ is free. The Galois cover $C \to C/G$ with $G=\bZ/p$ is ramified while the Galois cover $\tilde C \to \tilde C / G$ is \'etale. This shows that every arc on $C/G$ lifts to $C$ unless it maps $\eta \in \bD$ to the branch point. Observe that $C$ is not KLT nor ($W\sO$-)rational.
\end{example}

\subsection{Main result} We are now ready to establish our main result:

\begin{theorem}[Strict decent] \label{thm.StringyInvariantGoesDown} Work in \autoref{setup.MainTheorems}. Suppose that $(X_0,\Delta_0)$ is KLT of dimension $d\leq 3$. Assume that $g_1$ and $g_2$ are quasi-\'etale. In the case $d=3$ and $0<p \leq 5$, further assume that $G_1$ and $G_2$ are $p$-groups. If $f$ is not \'etale, then:
\[
M_{\st}(X_{1},\Delta_1)/G_{1}>M_{\st}(X_{2},\Delta_2)/G_{2}.
\]
Similarly, if $x_i\in X_i$ is the preimage of $x_0$ then 
\[
M_{\st}(X_{1},\Delta_1)_{x_1}/G_{1}>M_{\st}(X_{2},\Delta_2)_{x_2}/G_{2}.
\]
\end{theorem}

\begin{proof} Put together \autoref{cor.KeyCorollary} and \autoref{pro.SecondProposition}.
\end{proof}

\section{DCC for Stringy Motives of Log Terminal Surface Singularities} \label{sec.DCCSurfaces}

In this section, we discuss the \emph{Descending Chain Condition} (DCC) for stringy motives of log terminal surface singularities. We work in the setup of \autoref{sec.ExplicitDescriptionSurfaceCase}. We commence by using the formula \autoref{eqn.DescriptionOfStringyInvariantsForCurves} to see that, as a Laurent power series in $\bL^{-1/r}$, the invariant $M_{\st}(X)_x/G$ can be expressed as: 
\begin{multline*}
\sum_{\iota \in I/G}(\bL+m_{\iota})(\bL-1)(\bL^{-a_{\iota}}+\bL^{-2a_{\iota}}+\cdots)\\
+\sum_{ [\iota,\kappa] \in H/G }(\bL-1)^{2}(\bL^{-a_{\iota}}+\bL^{-2a_{\iota}}+\cdots)(\bL^{-a_{\kappa}}+\bL^{-2a_{\kappa}}+\cdots),
\end{multline*}
which has degree $2-\min\{a_{\iota}\}$, where we set $\deg\bL^{1/r}=1/r$. Modulo terms of degrees $<1$, the above series becomes
\begin{equation}
\sum_{\iota \in I/G}\bL^{2}(\bL^{-a_{\iota}}+\bL^{-2a_{\iota}}+\cdots)+\sum_{[\iota,\kappa] \in H/G}\bL^{2}(\bL^{-a_{\iota}}+\bL^{-2a_{\iota}}+\cdots)(\bL^{-a_{\kappa}}+\bL^{-2a_{\kappa}}+\cdots).\label{eq:mod deg < 1}
\end{equation}
This shows that $M_{\st}(X)_x/G$ has nonnegative coefficients for the terms $\bL^{b}$ with $1\le b<2$, which motivates the following definition.
\begin{definition}
With notation as above, we define $N(x,X,G)\in\bN[\bL^{1/r}]$ to be $M_{\st}(X)_x/G$ with all the terms of degree $<1$ removed. We define $C(x,X,G)\in\bN$ to be the sum of the coefficients of $N(x,X,G)$. 
\end{definition}

\begin{remark} \label{rem.Monotonicity}
Clearly, if $M_{\st}(X)_x/G \geq M_{\st}(X')_{x'}/G'$ then $N(x,X,G)\ge N(x',X',G')$. 
\end{remark}

Note that the expression \autoref{eq:mod deg < 1} shows that each vertex of $\Gamma/G$ with log discrepancy $\le 1$ contributes at least 1 to $C(x,X,G)$. From \autoref{lem:log discrep at most 1}, we then obtain:

\begin{lemma} \label{lem.BoundingVertices}
The graph $\Gamma/G$ has at most $C(x,X,G)$ non-special vertices (i.e., the vertices coming from the minimal resolution) and
at most $C(x,X,G)+1$ vertices.
\end{lemma}

\begin{definition}Let $r$ be a positive integer. We define $\sA_r^2 \subset \bZ \llparenthesis \bL^{-1/r} \rrparenthesis$ to be the subset of elements $M_{\st}(X)_x/G$ such that $X$ is a log terminal surface with an action of a finite group $G$ and such that $rK_X$ is Cartier and the $G$-action fixes a point $x\in X$.
\end{definition}
\begin{lemma}
\label{lem:bounded C(X,G)} Fix a positive integer $r$ and a polynomial $N \in \bN[\bL^{1/r}]$. Then, there exist only finitely many elements $\alpha \in \sA_r^2$ being equal to a stringy motive $M_{\st}(X)_x/G$ such that $N(x,X,G)=N$. 
\end{lemma}

\begin{proof}
The invariant $M_{\st}(X)/G$ is determined by the data of the dual graph $\Gamma/G$ associated to the modified minimal resolution of $X$ and the log discrepancies $a_{\iota}$ assigned to their vertices. Since $N(x,X,G)$
determines $C(x,X,G)$, the number of vertices of $\Gamma/G$ is bounded by \autoref{lem.BoundingVertices}. Thus, there are only finitely many possibilities for the graph $\Gamma/G$. For each possibility, there are only finitely many ways to assign numbers $a_{\iota}$ to vertices as these numbers must belong to the finite set $\frac{1}{r}\bZ\cap(0,2]$. The result then follows.
\end{proof}

\begin{proposition}[DCC for surface stringy motives] 
\label{prop:DCC} Fix a positive integer $r$. Then, $\sA_r^2$ satisfies DCC: every descending chain
\[
M_{\st}(X_{0})_{x_0}/G_{0}\ge M_{\st}(X_{1})_{x_1}/G_{1}\ge M_{\st}(X_{2})_{x_2}/G_{2} \ge \cdots
\]
of elements in $\sA_r^2$ eventually stabilizes. 
\end{proposition}

\begin{proof}According to \autoref{rem.Monotonicity}, from the chain in the statement, we obtain the following decreasing chain
\[
N(x_0, X_0, G_0) \geq N(x_1,X_1,G_1) \geq N(x_2,X_2,G_2) \geq \cdots.
\]
of polynomials in $\bN[\bL^{1/r}]$. Therefore, it must stabilize. By \autoref{lem:bounded C(X,G)}, the chain of the statement stabilizes as well.
\end{proof}

\begin{question}[DCC for threefold stringy motives] \label{que.DCCForThreefolds}
Fix a positive integer $r$. Let $\sA_r^3 \subset \hat{\sM}'_{\kay,r}$ be the subset of stringy motives $M_{\st}(X,\Delta)_x/G$ where $(X,\Delta)$ is a KLT $3$-dimensional log pair such that $r(K_X + \Delta)$ is Cartier and $G$ acts on $X$ fixing the boundary $\Delta$ as well as $x \in X$. Does $\sA_r^3$ satisfy the DCC? If $0<p \leq 5$, does DCC hold when we consider $p$-groups only? See \autoref{rem.StrongerStatement}. 
\end{question}

\section{Applications to \'Etale Fundamental Groups}

We finish by establishing some consequences of the above results. The main consequence is that \autoref{que.questionIntro.} has an affirmative answer if $(X,\Delta)$ is a KLT surface. Indeed:

\begin{theorem}
Let 
\[
X_0 \xleftarrow{f_0} X_1 \xleftarrow{f_1} X_2 \xleftarrow{f_2} X_3 \leftarrow \cdots
\]
be a tower of Galois quasi-\'etale covers of log terminal surfaces over an algebraically closed field. Then, $f_i$ is \'etale for all $i \gg 0$.
\end{theorem}
\begin{proof}
Suppose for the sake of contradiction that $f_i$ is not \'etale for infinitely many $i$. Let $G_i \coloneqq \Gal(X_i/X_0)$. 
Since each $f_i$ is \'etale over the regular locus of $X_0$, which has only finitely many singular points, there exists a singular point $x_0\in X_0$ such that for infinitely many $i$, $f_i$ is not \'etale over $x_0$. Shrinking $X_0$ around this point, we may suppose that $X_0$ has the unique singular point $x_0$.  
Note that, if $r$ is the Gorenstein index of $X_0$, then $r$ is a multiple of the Gorenstein index of every $X_{i}$. By \autoref{thm.StringyInvariantGoesDown}, we get a decreasing sequence
\[
M_{\st}(X_{0})/G_{0}\ge M_{\st}(X_{1})/G_{1}\ge M_{\st}(X_{2})/G_{2}\ge \cdots
\]
such that infinitely many inequalities are strict. 
For each $i$, let us choose a point $x_i\in X_i$ lying over $x_0$. Then,
\[
M_{\st}(X_i)/G_i = \{X_0 \setminus \{x_0\}\} +  M_{\st}(X_i)_{x_{i}}/\Stab(x_{i})
\]
by using \autoref{eqn.GlobalvsLocalSurfaceCaseManyPoints}. Hence, by subtracting $\{X_0 \setminus \{x_0\}\}$ from the above sequence, we obtain a non-terminating decreasing sequence 
\[
M_{\st}(X_0)_{x_{0}}/\Stab(x_{0}) \ge
M_{\st}(X_1)_{x_{1}}/\Stab(x_{1}) \ge
M_{\st}(X_2)_{x_{2}}/\Stab(x_{2}) \ge \cdots
\]
of elements in the set $\sA_r^2\subset\bZ \llparenthesis \bL^{-1/r} \rrparenthesis$ (which was defined in \autoref{sec.DCCSurfaces}). This contradicts \autoref{prop:DCC}.
\end{proof}

\begin{remark}
A little more effort enables us to prove the above theorem by the DCC of the sequence $M_{\st}(X_i)/\Stab(x_{i})$ without shrinking $X_0$ (but still reducing to the DCC of the ``local'' stringy motives $M_{\st}(X_i)_{x_i}/\Stab(x_{i})$). This approach would be a more faithful practice of the strategy explained in \autoref{sec.Introduction}.
\end{remark}

A similar (but simpler) argument shows the following; see \autoref{cor.FinitenessFUndGrouosMixChar} and its proof. 

\begin{theorem} \label{thm.ApplicationFundamentalGroups}
The local \'etale fundamental group of a log terminal surface germ is finite.
\end{theorem}
As an immediate consequence, we obtain:
\begin{corollary}
Big opens of weak del Pezzo surfaces have finite \'etale fundamental group.
\end{corollary}
\begin{proof}
See \cite[Theorem 2.6]{CarvajalRojasStablerKollarFundamentalGroupKLTthreefolds} and its proof.
\end{proof}

\begin{remark}\label{rem.FundamentalGropuSurfaceCase}
To the best of the author's knowledge, the results of this section were unknown at least in characteristics $2$ and $3$. In characteristic $p>5$, one may use that log terminal surfaces are strongly $F$-regular and use \cite{CarvajalSchwedeTuckerEtaleFundFsignature,BhattCarvajalRojasGrafSchwedeTucker}. In characteristic $p=5$, one may use that the canonical cover of log terminal surface singularities is a rational point; see \cite{KawamataIndex1CoversOfLogTerminalSurfaceSingularities,ArimaOnIndexOneCoversOfTwoDimensionalPLTSingsInPosChar}, and then use Artin's explicit description of them in \cite{ArtinCoveringsOfTheRtionalDoublePointsInCharacteristicp} to conclude. In characteristics $p=2,3$, we know the canonical cover trick fails by \cite{KawamataIndex1CoversOfLogTerminalSurfaceSingularities} hence a new, different approach was needed. Our methods addressed these cases with the salient feature of providing a unified approach which is conceptual and independent of the characteristic of the groundfield. For instance, it gives a conceptual proof for the finiteness of the \'etale local fundamental group of rational double points in the low characteristics, which Artin had accomplished only on a case-by-case classification analysis \cite{ArtinCoveringsOfTheRtionalDoublePointsInCharacteristicp}. 
\end{remark}

\section{Mixed Characteristic Surface Case} \label{sec.Mixed_characteristic_surface case}

In this section, we aim to establish \autoref{thm.ApplicationFundamentalGroups} in the mixed characteristic case as well. We may use the exact same proof as long as we establish \autoref{thm.StringyInvariantGoesDown} in mixed characteristics. The first problem we face with this is that there is no suitable theory of motivic integration in mixed characteristic; see \cite[Problem 10.1]{YasudaOpenProblems}. Therefore, we cannot see the strict descent of stringy motives formally as the non-liftability of arcs as we did in \autoref{sec.StrictDecentRamification}. However, the surface case is simple enough for this descent to be seen by direct analysis, which is what we do in this section.

Let $(R,\fram,\kay,K)$ be a $2$-dimensional complete normal integral domain with algebraically closed residue field $\kay$ and field of fractions $K$. Suppose that it has mixed characteristic $(\Char K =0,\Char\kay = p>0)$. In that case, $R$ admits a canonical module which is a reflexive module of rank $1$ and unique up to isomorphism. For instance, we may use that $R$ is a finite extension over $A= \Lambda\llbracket t \rrbracket $ where $(\Lambda, (p), \kay)$ is a complete DVR (see \cite[\href{https://stacks.math.columbia.edu/tag/032D}{Tag 032D}]{stacks-project}) to set $\omega_R \coloneqq \omega_{R/A} \coloneqq \Hom_A(R,A)$. We let $K_X$ denote a corresponding canonical divisor on $X \coloneqq \Spec R$. The choice of $K_X$ is unique up to linear equivalence and so the canonical class $K_X\in \Cl X$ is well-defined. Further, we write $x \coloneqq \Spec \kay$ and identify it with the closed point of $X$. 

We will assume that $X$ is $\bQ$-Gorenstein, i.e. $K_X \in \Cl X$ is torsion, and denote by $r \geq 1$ the index of $K_X$ in $\Cl X$. In particular, we may define log discrepancies with respect to a minimal resolution and define $(x,X)$ as log terminal if these are positive. In doing this, we are setting our base to be $A=\Lambda\llbracket t \rrbracket$; with notation as in the previous paragraph. 

\begin{remark}
Sometimes one considers a relative canonical module $\omega_{R/B}$ with respect to a different base $B$. For example, one may consider a scheme $Y$ of finite type over a complete DVR $B$ with an algebraically closed residue field and suppose that $R$ is the complete local ring $\hat{\sO}_{Y,y}$ at a closed point $y$. Then, the relative canonical module is defined as $\omega_{R/B}=\omega_{Y/B}\otimes \hat{\sO}_{Y,y}$. However, whether we use  $\omega_R$ or $\omega_{R/B}$ does not make any change to the following argument; for we get the same discrepancies in either way. See ``Comments'' on pages 7 and 8 of  \cite{KollarSingulaitieofMMP}. In what follows, all canonical divisors are defined with respect to $A$ as above. However, as in \cite{KollarSingulaitieofMMP}, we will omit the base in our notation for canonical divisors.
\end{remark}

Suppose that $(x,X)$ is a log terminal singularity and let $\phi \colon\tilde{X} \to (x,X)$ be the minimal resolution. Then, we \emph{define} $M_{\st}(X)_x$ by the explicit formula in \autoref{ex.Surfaces}. Namely,
\[
M_{\st}(X)_x \coloneqq \sum_{i\in I} (\bL +1 - m_i) \frac{\bL-1}{\bL^{a_{i}}-1}+\sum_{[i,j]\in H}\frac{(\bL-1)^{2}}{(\bL^{a_i}-1)(\bL^{a_j}-1)} \in \bZ\llparenthesis\bL^{-1/r} \rrparenthesis \subset \hat{\sM}'_{\kay,r},
\]
where we use the exact same notation as in \autoref{ex.Surfaces}. Recall that $a_i \in\frac{1}{r}\bZ \cap (0,1]$ as $(x,X)$ is log terminal with Gorenstein index $r$. 
 
Likewise, if $G$ is a finite (discrete) group acting on $X$ and fixing $x$, we may further define the quotient stringy invariant $M_{\st}(X)_x/G$ using the modified minimal log resolution as explained in \autoref{sec.ExplicitDescriptionSurfaceCase}. That is, we take \autoref{eqn.DescriptionOfStringyInvariantsForCurves} as the definition of $M_{\st}(X)_x/G$ in mixed characteristics:
\[
M_{\st}(X)_x/G \coloneqq \sum_{\iota\in I/G} (\bL +1 - m_{\iota}) \frac{\bL-1}{\bL^{a_{\iota}}-1}+\sum_{[\iota,\kappa]\in H/G}\frac{(\bL-1)^{2}}{(\bL^{a_{\iota}}-1)(\bL^{a_{\kappa}}-1)} \in \bZ\llparenthesis\bL^{-1/r} \rrparenthesis \subset \hat{\sM}'_{\kay,r}.
\]

For the sake of concreteness, we have used minimal resolutions and minimal $G$-normal resolutions to define the stringy motives above. However, using strong factorization, we may take any log resolution as the following lemma makes precise.
\begin{lemma}\label{lemma.StringyMotivesAreWellDefined}
With notation as above, let $\psi \colon Y \to (x,X)$ be a log resolution and $\Gamma$ be the corresponding dual graph with set of vertices $I$ and set of edges $H$. Further, assume that $E_i$ has log discrepancy $a_i$ for all $i \in I$. Then,
\[
M_{\st}(X)_x=\sum_{i\in I} (\bL +1 - m_i) \frac{\bL-1}{\bL^{a_{i}}-1}+\sum_{[i,j]\in H}\frac{(\bL-1)^{2}}{(\bL^{a_i}-1)(\bL^{a_j}-1)},
\]
where $m_i$ is the number of edges of $\Gamma$ sticking out of $i$. Suppose further that $\psi$ is a $G$-normal log resolution, then
\[
M_{\st}(X)_x/G=\sum_{\iota\in I/G} (\bL +1 - m_{\iota}) \frac{\bL-1}{\bL^{a_{\iota}}-1}+\sum_{[\iota,\kappa]\in H/G}\frac{(\bL-1)^{2}}{(\bL^{a_{\iota}}-1)(\bL^{a_{\kappa}}-1)},
\]
where $m_{\iota}$ is the number of edges of $\Gamma/G$ sticking out of $\iota$ and $a_{\iota} = a_i$ for all $i \in I$.
\end{lemma}
\begin{proof}
Any two log resolutions are related by a sequence of blowups at a point (strong factorization). Thus, it suffices to consider $\theta \colon Y' \to Y$ the blowup of $Y$ at an edge of $\Gamma$ and prove that our formula remains invariant when computed using the log resolution $Y' \to Y \to (x,X)$. Note that the dual graph of $Y'$, say $\Gamma'$, is obtained from $\Gamma$ by replacing the edge that is blown up, say $[i,j]$, by $[i,k]\cup [k,j]$ where $k$ is the new vertex of $\Gamma'$ corresponding to the exceptional divisor of $Y' \to Y$, whose log discrepancy over $X$ is $a_k=a_i+a_j$. All other discrepancies remain unchanged. In particular, it suffices to prove that
\[
(\bL-1) \frac{\bL-1}{\bL^{a_k}-1} + \frac{(\bL-1)^2}{(\bL^{a_i}-1)(\bL^{a_k}-1)}+\frac{(\bL-1)^2}{(\bL^{a_k}-1)(\bL^{a_j}-1)} = \frac{(\bL-1)^2}{(\bL^{a_i}-1)(\bL^{a_j}-1)},
\]
which is a straightforward consequence of the equality $a_k=a_i+a_j$.
\end{proof}

We aim to prove the following.

\begin{theorem} \label{thm.MainTheoremMixChar}
Let $(x,X)$ be a log terminal surface singularity as above and $G$ be a nontrivial finite group acting on $X$ and fixing $x$. Let $\bar{x}$ be the image of $x$ on the quotient $X/G$. Assume that the quotient morphism $f\colon (x,X) \to (\bar{x},X/G)$ is \'etale away from $\bar{x}$. Then,
\[
M_{\st}(X/G)_{\bar{x}}>M_{\st}(X)_x/G,
\]
with respect to the lexicographic order in $\bZ\llparenthesis \bL^{-1/r}\rrparenthesis$.
\end{theorem}

As a corollary, we obtain the following.

\begin{corollary} \label{cor.FinitenessFUndGrouosMixChar}
Let $(x,X)$ be a log terminal surface singularity as above. Then, the \'etale fundamental group $\pi_1^{\mathrm{\Acute{e}t}}(X\setminus \{x\})$ is finite.
\end{corollary}
\begin{proof}
The same arguments in \autoref{sec.DCCSurfaces} shows that the set $\sA^2_{\kay,r} \subset \hat{\sM}'_{\kay,r}$ of stringy motives $M_{\st}(X)_x/G$ such that $(X,x)$ is a log terminal singularity with $rK_X = 0 \in \Cl X$ and admitting an action of a finite group $G$ fixing $x=\Spec{\kay}$ satisfies a descending chain condition (DCC).

If $\pi_1^{\mathrm{\Acute{e}t}}(X\setminus \{x\})$ were not finite, then there would be an infinite chain of finite Galois local covers $(x,X) \leftarrow (x_1,X_1) \leftarrow (x_2,X_2) \leftarrow \cdots $ such that $X_{i+1}\setminus \{x_{i+1}\} \to X_i \setminus \{x_i\}$ is \'etale but $X_{i+1} \to X_i$ is not \'etale. In particular, the Galois group $1\neq \Gal(X_i/X)$ acts on $X_i$ fixing $x_i$ and the quotient is $(x,X)$. Further, every $(x_i,X_i)$ is a log terminal surface singularity with $x_i = \Spec \kay$ and such that $rK_{X_i} = 0\in \Cl X_i$. Then, by applying \autoref{thm.MainTheoremMixChar}, we obtain an strictly ascending chain
\[
M_{\st}(X)_x > M_{\st}(X_1)_{x_1}/G_1>M_{\st}(X_2)_{x_2}/G_2 > \cdots
\]
violating the DCC on $\sA^2_{\kay,r}$.
\end{proof}

The rest of this section is dedicated to the proof of \autoref{thm.MainTheoremMixChar}. Since we cannot use motivic integration to compare $M_{\st}(X/G)_{\bar{x}}$ and $M_{\st}(X)_x/G$, we have to compare the explicit formulas defining them. However, the main difficulty in making this comparison is that the $G$-quotient of a $G$-normal log resolution $\phi \colon \Tilde{X} \to (x,X)$ is not a log resolution of $X/G$. Our strategy to bypass this issue is to compare stringy motives along some constructible subsets and then take limits. To this end, we shall need some auxiliary stringy motives along closed subsets. First, we explain the setup in which discrepancies can be compared directly.

\subsection{Comparing discrepancies}
Consider the following commutative square between normal integral schemes
\[
\xymatrix{
\Tilde{W} \ar[d]_-{\phi} & \Tilde{X} \ar[d]^-{\psi} \ar[l]_-{\Tilde{f}} \\
(w,W) & \ar[l]_-{f} (x,X)
}
\]
where $(w,W)$ is a log terminal surface singularity, $f$ is a quasi-\'etale cover of surface singularity germs, and the vertical arrows are proper birational morphisms. When $f$ is a Galois cover, we also assume that the Galois action on $X$ lifts to $\Tilde{X}$.
Let $F\subset \Tilde{X}$ be a prime divisor over $X$ and let $E\subset \Tilde{W}$ be its image. Suppose that these divisors have (non-log) discrepancies $b_F$ and $b_E$ over $X$ and $W$; respectively. 

\begin{proposition}\label{prop:compare-discreps}
With notation as above, suppose that either 
    \begin{enumerate}
        \item the free $\Hat{\sO}_{\Tilde{W},\eta_E}$-module $\Hat{\sO}_{\Tilde{X},\eta_F}$ has prime-to-$p$ rank, or
        \item $f$ is a Galois cover.
    \end{enumerate}
Then, $b_F\ge b_E$. Moreover, if $\Tilde{f}$ is ramified along $F$, then $b_F > b_E$. 
\end{proposition}

\begin{proof}
Concerning the divisors appearing in this proof, we are mainly interested in the coefficients of $F$ and $E$ so that we will omit other prime divisors by writing ``$\cdots$''. 

We can write 
\[
K_{\Tilde{X}}= \psi^* K_X + b_F F + \cdots
\]
and 
\[
K_{\Tilde{W}}= \phi^* K_W + b_E E + \cdots.
\]
Let $e$ be the ramification index at $F$. Namely, the uniformizer of $\Hat{\sO}_{\Tilde{W},E}$ has order $e$ with respect to the normalized valuation of $\Hat{\sO}_{\Tilde{X},F}$. Let $\delta$ denote the different at $F$, which is defined to be the length of the $\Hat{\sO}_{\Tilde{X},F}$-module $\hat{\Omega}_{\Tilde{X}/\Tilde{W},F}$. 
Since $f$ is quasi-\'etale, pulling back the second equality gives
\[
K_{\Tilde{X}} = \Tilde{f}^*K_{\Tilde{W}} + \delta F+\cdots 
= \psi ^* K_X + (eb_E+\delta) F+\cdots.
\]
 Comparing the coefficients of $F$, we get
\[
b_F = eb_E + \delta. 
\]

\begin{proof}[Proof of case (a)]
In this tame case, it is well-known that $\delta=e-1$. It follows that 
\[
b_F+1=e(b_E+1) \ge b_E +1. 
\]
Moreover, if $\Tilde{f}$ is ramified along $F$, then $e>1$ and hence $b_F+1 > b_E +1$. The assertion follows.
\end{proof}

\begin{proof}[Proof of case (b)]
From Hyodo's formula \cite[(1-4)]{HyodoWildImperfect}, we have
\[
\delta = e -1 + d_L(M/L),
\]
where $M$ and $L$ are the fraction fields of $\Hat{\sO}_{\Tilde{X},\eta_F}$ and $\Hat{\sO}_{\Tilde{W},\eta_E}$, and $d_L(M/L) \in \bN$ is the depth of ramification of the extension $M/L$. We have that $d_L(M/L)=0$ if and only if $\Tilde{f}$ is tamely ramified along $F$. Let $G=\Gal(L/K)$ be the Galois group.

 \emph{The case $G\cong \bZ/p$:} Let $k_F$ and $k_E$ be the residue fields at $F$ and $E$ respectively. This case is further divided into three subcases according to ramification type: unramified ($e=1$ and $k_F/k_E$ is separable), wildly ramified ($e=p$ and $[k_F:k_E]=1$), and fiercely ramified ($e=1$ and $k_F/k_E$ is inseparable). The unramified case is contained in the tamely ramified case (a).
 In the other two cases, from \cite[Section 2.1]{XiaoZhukovRamifLocalFields}, $\delta \ge p-1$.
 In the wildly ramified case, 
 \[
b_F = pb_E +\delta \ge 
p(b_E+1)-1,
 \]
 and
 \[
 b_F+1 > b_E+1,
 \]
as $b_E+1>0$ from the log terminal condition. In the fiercely ramified case, 
 \[
 b_F-b_E = \delta =d_L(M/L)>0. 
 \]
 Thus we get the desired assertion in every subcase.
 
 \emph{The case $G$ is a $p$-group:} Suppose that $G$ has order $p^{a}$ with $a>0$. Let $C\subset G $ be a central group of order $p$. It is well-known that every $p$-group admits such a subgroup. Suppose that the assertion holds when the Galois group is a $p$-group of order $\le p^{a-1}$. Let $F'$ be the image of $F$ in $\Tilde{X}/C$. From the degree-$p$ case and the induction hypothesis, we get 
 \[
 b_E \le b_{F'} \le b_F.
 \]
 If $\Tilde{X}\to \Tilde{W}$ is ramified along $F$, then either $\Tilde{X}\to \Tilde{X}/C$ is ramified along $F$ or $\Tilde{X}/C\to \Tilde{W}$ is ramified along $F'$. Thus, at least one of the above intermediate inequalities is strict, and hence $b_E<b_F$, if $\Tilde{X}\to \Tilde{W}$ is ramified along $F$. We have proved the assertion in the case of $p$-groups.
 
 \emph{The general Galois case:} Let $H \subset G$ be the stabilizer of $F$ and $S\subset H$ be a Sylow $p$-subgroup. Let $F'$ be the image of $F$ in $\Tilde{X}/S$. From the $p$-group case and the tame (non-Galois) case, we get 
 \[
 b_E \le b_{F'} \le b_F.
 \]
In the ramified case, we get the strict inequality $b_E<b_F$ as in the last case. We have completed the proof of (b).
\end{proof}
This concludes the proof of \autoref{prop:compare-discreps}.
\end{proof}

\subsection{Tweak by constructible subsets}

In this subsection, we define a version of stringy motives, denoted by $M_{\st}(X)_{x,C}/G$, in the following situation. We consider a $G$-equivariant proper birational morphism $\phi \colon \Tilde{X} \to (x,X)$ and a $G$-invariant constructible subset $C \subset \phi^{-1}(x)$ such that
$(\Tilde{X},\phi^{-1}(x))$ is a $G$-normal SNC pair in a neighborhood of $C$. We write $\phi^{-1}(x)=\bigcup_{i\in I} E_i$ as usual, where $E_i$ are prime divisors. Following the notation in  \autoref{lem.SurfaceCaseDescription} and the paragraph before it, 
we define:
\begin{equation}\label{eq.def-formula}
\begin{split}
M_{\st}(X)_{x,C}/G 
\coloneqq & \sum_{\iota\in I/G}\{(C \cap E_\iota^\circ)/G\}\frac{\bL-1}{\bL^{a_{\iota}}-1} \\ 
&+\sum_{\{\iota, \kappa\}\subset I/G}\{(C\cap E_\iota\cap E_\kappa)/G\}\frac{(\bL-1)^{2}}{(\bL^{a_{\iota}}-1)(\bL^{a_{\kappa}}-1)}
\end{split}
\end{equation}
Note that curves $E_i$ as well as their quotients by finite group actions are all rational. Thus, their classes in our Grothendieck ring are of the form $\bL+n$ with $n$ an integer. Therefore, the above invariant belongs to $\bZ\llparenthesis\bL^{-1/r} \rrparenthesis $ with $r$ the Gorenstein index of $X$ or any multiple of it. 
It should be noted that if $\phi$ is a $G$-normal log resolution, then
\[
M_{\st}(X)_{x}/G = M_{\st}(X)_{x,\phi^{-1}(x)}/G.
\]
As a special case, we define $M_{\st}(X)_{x,C} \coloneqq M_{\st}(X)_{x,C}/\{e\}$ where $\{e\}$ is the trivial group. Thus, we may readily generalize \autoref{lemma.StringyMotivesAreWellDefined} as follows:
\begin{lemma}
Consider the following commutative diagram of $G$-equivariant proper birational morphisms:
\[
\xymatrix{
\Tilde{X} \ar[dr]_-{\phi} & & \Tilde{X}' \ar[dl]^-{\phi'} \ar[ll]_-{\psi}\\
& (x,X) &
}
\]
Let $C \subset \phi^{-1}(x)$ be a $G$-invariant constructible subset. Suppose that $(\Tilde{X},\phi^{-1}(x))$ is a $G$-normal SNC pair in a neighborhood of $C$ and $(\Tilde{X}',\phi'^{-1}(x))$ is a $G$-normal SNC pair in a neighborhood of $\psi^{-1}(C)$. Then, the following equality holds:
\[
M_{\st}(X)_{x,C}/G = M_{\st}(X)_{x,\psi^{-1}(C)}/G.
\]
\end{lemma}

The following lemma is a direct consequence of the definition.

\begin{lemma}
Let $\phi \colon \Tilde{X} \to (x,G)$ be a $G$-equivariant proper birational morphism and let $C_{\lambda} \subset \phi^{-1}(x)$ be pairwise disjoint $G$-invariant constructible subsets. Suppose that $(\Tilde{X},\phi^{-1}(x))$ is a $G$-normal SNC pair in a neighborhood of $\bigsqcup_\lambda C_\lambda$. Then
\[
M_{\st}(X)_{x,\bigsqcup_{\lambda} C_{\lambda}}/G = \sum_{\lambda} M_{\st}(X)_{x,C_{\lambda}}/G.
\]
\end{lemma}

\subsection{Approximation by constructible subsets}
The following construction will let us approximate $M_{\st}(X/G)_{\bar{x}}$ and $M_{\st}(X)_x/G$ by stringy motives of the form $M_{\st}(X/G)_{\bar{x},\overline{C}}$ and $M_{\st}(X)_{x,C}/G$ where the inequality $M_{\st}(X/G)_{\bar{x},\overline{C}} \geq M_{\st}(X)_{x,C}/G$ holds.
We first construct the following diagram in a way explained below:
\begin{equation}\label{dia:ladder}
\xymatrix@C=3em{
X=X_0 \ar[d]^-{f} & X_i \ar[d] \ar@{..>}[l] & Y_i \ar[d] \ar[l] & X_{i+1} \ar[d] \ar[l] &  \ar@{..>}[l]  \\
X/G=X_0/G & X_i/G \ar@{..>}[l] & Y_i/G \ar[l] & X_{i+1}/G \ar[l] &  \ar@{..>}[l]
}
\end{equation}
Suppose that we have constructed up to $X_i$ and $X_i/G$. We then take a log resolution $Y_i/G\to X_i/G$ and define $Y_i$ as the normalization of the fiber product $X_i \times_{X_i/G} Y_i/G$ so that $Y_i/G$ is the quotient of $Y_i$ by the induced $G$-action as the notation suggests. Then, we take a $G$-normal log resolution $X_{i+1} \to Y_i$ and define $X_{i+1}/G$ to be its quotient. We suppose that the log resolutions $Y_i/G\to X_i/G$ and $X_{i+1} \to Y_i$ are isomorphisms over the SNC loci of the targets paired with the exceptional loci of $X_i/G\to X/G$ and $Y_i \to X$; respectively. Repeating this procedure produces a diagram as above. 

Let $B_i\subset X_i$ be the exceptional locus of $X_i \to X$ and let $\bar{B}_i\subset X_i/G$ be its image. 
Let $\bar{C}_i\subset \bar{B}_i$ be the largest open subset along which the pair $(X_i/G,\bar{B}_i)$ is SNC.  Let $C_i$ be the preimage of $\bar{C}_i$ in $X_i$, which is an open subset of $B_i$. From the construction, $X_{i+1}\to X_i$ is an isomorphism over $C_i$. Observe that $\bar{B}_i \setminus \bar{C}_i$ and ${B}_i \setminus {C}_i$ are finite sets of closed points. 

\begin{lemma}\label{lem.five-assertions}
With the above notation, the following statements hold:
\begin{enumerate}
    \item $M_{\st}(X/G)_{\bar{x},\bar{C}_i} \geq M_{\st}(X)_{x,C_i}/G$
    \item $M_{\st}(X/G)_{\bar{x}} = \lim_{i \rightarrow \infty} M_{\st}(X/G)_{\bar{x},\bar{C}_i}$
    \item $M_{\st}(X)_{x}/G = \lim_{i \rightarrow \infty} M_{\st}(X)_{x,C_i}/G$
    \item $M_{\st}(X/G)_{\bar{x}} \geq M_{\st}(X)_{x}/G$
    \item If $X_i \to X_i/G$ is ramified along a prime divisor $E\subset X_i$, then $M_{\st}(X/G)_{\bar{x},\bar{C}_i} > M_{\st}(X)_{x,C_i}/G$. Moreover, if $E$ has log discrepancy $a$ relative to $X$, then
    \[
    \dim M_{\st}(X/G)_{\bar{x},\bar{C}_i} - M_{\st}(X)_{x,C_i}/G \ge 2-a.
    \]
\end{enumerate}
\begin{proof}
The statement (a) follows from the defining formula \autoref{eq.def-formula} of a stringy motive and the comparison of discrepancies in \autoref{prop:compare-discreps}.

For (b), we first note that the ring $\hat{\sM}'_{\kay,r}$, where our stringy motives live, is the completion of a ring denoted by $\sM'_{\kay,r}$ with respect to a descending filtration $\{F_m\}_{m\in (1/r)\bZ}$ \cite[p.~222]{YasudaMotivicIntegrationDeligneMumfordStack}. Hence, $\hat{\sM}'_{\kay,r}$ is complete with respect to the filtration $\{\hat F_m\}_{m\in (1/r)\bZ}$, where $\hat F_m := \varprojlim _{m'\ge m} F_m /F_{m'}$. Thus, to show (b), it suffices to show that for any $m$, if $i$ is sufficiently large, then 
\[
\big( M_{\st}(X/G)_{\bar{x}}-M_{\st}(X/G)_{\bar{x},\bar{C}_{i+1}} \big) \in \hat F_m. 
\]
In turn, proving it is reduced to proving that for each $i$, 
\begin{equation}\label{eq.2}
\dim \big( M_{\st}(X/G)_{\bar{x}}-M_{\st}(X/G)_{\bar{x},\bar{C}_{i+1}} \big) < 
\dim \big(M_{\st}(X/G)_{\bar{x}}-M_{\st}(X/G)_{\bar{x},\bar{C}_i}\big).
\end{equation}
Since the preimage of $\bar{C}_i$ in $X_{i+1}/G$ is contained in  $\bar{C}_{i+1}$,
\[
M_{\st}(X/G)_{\bar{x},\bar{C}_{i+1}}\ge M_{\st}(X/G)_{\bar{x},\bar{C}_i}.
\]
Let $D\subset Y_i/G$ be the preimage of $\bar{B}_i\setminus \bar{C}_i$, which has pure dimension 1. We have
\[
M_{\st}(X/G)_{\bar{x}}-M_{\st}(X/G)_{\bar{x},\bar{C}_i} = M_{\st}(X/G)_{\bar{x},D}.
\]
Let $D^\circ$ be the open dense subset of $D$ obtained by removing the image of 
$\bar{B}_{i+1}\setminus \bar{C}_{i+1}$ from $D$. Since $\bar{C}_i$ and $D^\circ$ are disjoint and the preimage of their union in $X_{i+1}/G$ is contained in $\bar{C}_{i+1}$, we have
\[
M_{\st}(X/G)_{\bar{x},\bar{C}_{i+1}} \ge M_{\st}(X/G)_{\bar{x},\bar{C}_{i}} + M_{\st}(X/G)_{\bar{x},D^\circ}.
\]
Since the preimage of $D\setminus D^\circ$ in $X_{i+1}/G$ contains $\bar{B}_{i+1}\setminus \bar{C}_{i+1}$, we have
\[
M_{\st}(X/G)_{\bar{x}}-M_{\st}(X/G)_{\bar{x},\bar{C}_{i+1}}
=M_{\st}(X/G)_{\bar{x},\bar{B}_{i+1}\setminus \bar{C}_{i+1}} 
\le  M_{\st}(X/G)_{\bar{x},D\setminus D^\circ}.
\]
Note that for every $y\in D \setminus D^\circ$, 
\begin{equation}\label{eq.1}
    \dim M_{\st}(X/G)_{\bar{x},y} < \dim M_{\st}(X/G)_{\bar{x},D}.
\end{equation}
If $y$ is a smooth point of $B_{i+1}$ and if $a$ denotes the log discrepancy at the prime divisor containing $y$, then for some integer $n$, we have
\[
\dim M_{\st}(X/G)_{\bar{x},y}= \dim \frac{\bL-1}{\bL^{a}-1} < \dim (\bL + n) \frac{\bL-1}{\bL^{a}-1} \le \dim M_{\st}(X/G)_{\bar{x},D}.
\]
If $y$ is a node of $B_{i+1}$ and if $a$ and $b$ are the log discrepancies at the two prime divisors containing $y$, then for some integer $n$, we have
\[
\dim M_{\st}(X/G)_{\bar{x},y}= \dim \frac{(\bL-1)^2}{(\bL^{a}-1)(\bL^{b}-1)} < \dim (\bL + n) \frac{\bL-1}{\bL^{a}-1} \le \dim M_{\st}(X/G)_{\bar{x},D}.
\]
Thus, in either case, \autoref{eq.1} is valid, and \autoref{eq.2} follows as desired.
Point (c) can be shown as in (b). 

To show the  statement (d), fixing $i$, we write $B_i = \bigcup_{j\in I} E_j$. Then, we have $\bar{B}_i = \bigcup_{\iota \in I/G} \bar{E}_\iota$, where $\bar{E}_\iota$ is the image of $E_\iota$ in $X_i/G$.
Note that the natural morphism $E_\iota/G\to \bar{E}_\iota$ is a universal homeomorphism.
Let $\bar{E}_\iota^\circ:=\bar{E}_\iota \setminus \bigcup _{\kappa \ne \iota} \bar{E}_\kappa$ and let $a_\iota$ (resp.\ $\bar{a}_\iota$) be the discrepancy at $E_\iota$ (resp.\ $\bar{E}_\iota$) relative to $X$ (resp.\ $X/G$).
Then, from the definition, $M_{\st}(X)_{x,C_i}/G$ is written as
\begin{align}\label{eq.surf1}
 \sum_{\iota\in I/G}\{(C_i \cap E_\iota^\circ)/G\}\frac{\bL-1}{\bL^{a_{\iota}}-1} 
+\sum_{\{\iota, \kappa\}\subset I/G}\{(C_i\cap E_\iota\cap E_\kappa)/G\}\frac{(\bL-1)^{2}}{(\bL^{a_{\iota}}-1)(\bL^{a_{\kappa}}-1)} ,
\end{align}
while $M_{\st}(X/G)_{x,\bar{C}_i}$ is written as
\begin{equation}\label{eq.surf2}
\begin{aligned}
& \sum_{\iota\in I/G}\{\bar{C}_i \cap \bar{E}_\iota^\circ\}\frac{\bL-1}{\bL^{\bar{a}_{\iota}}-1} 
+\sum_{\{\iota, \kappa\}\subset I/G}\{\bar{C}_i\cap \bar{E}_\iota\cap \bar{E}_\kappa\}\frac{(\bL-1)^{2}}{(\bL^{\bar{a}_{\iota}}-1)(\bL^{\bar{a}_{\kappa}}-1)}\\
& = \sum_{\iota\in I/G}\{(C_i \cap E_\iota^\circ)/G\}\frac{\bL-1}{\bL^{\bar{a}_{\iota}}-1} 
+\sum_{\{\iota, \kappa\}\subset I/G}\{(C_i\cap E_\iota\cap E_\kappa)/G\}\frac{(\bL-1)^{2}}{(\bL^{\bar{a}_{\iota}}-1)(\bL^{\bar{a}_{\kappa}}-1)}.
\end{aligned}
\end{equation}
Here, the last equality holds, for the maps $(C_i \cap E_\iota ^\circ)/G\to \bar{C}_i\cap \bar{E}_\iota^\circ$ are universal homeomorphisms.
Thus, the only difference in these formulas for $M_{\st}(X)_{x,C_i}/G$ and $M_{\st}(X/G)_{x,\bar{C}_i}$ is the exponents $a_\iota$ and $\bar{a}_\iota$. From \autoref{prop:compare-discreps}, we have $a_\iota \ge \bar{a}_\iota$,  which shows that $M_{\st}(X/G)_{x,\bar{C}_i}\ge M_{\st}(X)_{x,C_i}/G$. Taking  limits as $i\to \infty$, we get the statement (d). 

As noted just before this lemma, $B_i\setminus C_i$ consists of finitely many points and hence contains the generic point of every irreducible component of $B_i$, in particular, the generic point of a prime divisor where $X_i \to X_i/G$ is ramified. 
From \autoref{prop:compare-discreps}, we have strict inequality $a = a_\iota > \bar{a}_\iota $ for $\iota \in I/G$ such that $E\subset E_\iota$. 
From \autoref{eq.surf1} and \autoref{eq.surf2}, we have $M_{\st}(X/G)_{x,\bar{C}_i}> M_{\st}(X)_{x,C_i}/G$ and
\[
\dim\big(M_{\st}(X/G)_{\bar{x},\bar{C}_i} - M_{\st}(X)_{x,C_i}/G\big) \ge 
\dim \{(C_i \cap E_\iota^\circ)/G\}\frac{\bL-1}{\bL^a-1} = 2-a.
\]
We have proved the last statement (e). 
\end{proof}
\end{lemma}

\subsection{Proof of \autoref{thm.MainTheoremMixChar}}

In this subsection, we prove \autoref{thm.MainTheoremMixChar}. 
Let $(x,X)\to (\bar{x},X/G)$ as in the theorem. We then construct diagram \autoref{dia:ladder} and follow the notation there.

\begin{lemma}
    There exists a prime divisor $F \subset \bar{B}_1  \subset Y_1/G$ along which $Y_1\to Y_1/G$ is ramified.
\end{lemma}

\begin{proof}
 This is basically \autoref{pro.FirstKeyproposition} for surfaces in mixed characteristic. The same proof as ``Elementary proof of \autoref{pro.FirstKeyproposition} for surfaces'' shows the lemma. 
\end{proof}

We fix a divisor $F$ as in the last lemma. Let $a>0$ be its log discrepancy. For every $i$, the map $X_i \to X_i/G$ is ramified along the strict transform of $F$. 
From (d) of \autoref{lem.five-assertions}, 
\[
\dim\big(M_{\st}(X/G)_{\bar{x},\bar{C}_i} - M_{\st}(X)_{x,C_i}/G\big) \ge 
 2-a.
\]
On the other hand, 
from (b) and (c) of \autoref{lem.five-assertions}, if we take sufficiently large $i$, then 
\begin{align*}
2-a &>\dim \big(M_{\st}(X/G)_{\bar{x}}-M_{\st}(X/G)_{\bar{x},\bar{C}_i} \big),\\
2-a &>\dim \big(M_{\st}(X)_{x}/G-M_{\st}(X)_{x,C_i}/G \big).
\end{align*}
These inequalities together with 
\begin{align*}
M_{\st}(X/G)_{\bar{x}} - M_{\st}(X)_{x}/G
={}&
(M_{\st}(X/G)_{\bar{x},\bar{C}_i} - M_{\st}(X)_{x,C_i}/G) \\
&+(M_{\st}(X/G)_{\bar{x}}-M_{\st}(X/G)_{\bar{x},\bar{C}_i} ) \\
&+(M_{\st}(X)_{x}/G-M_{\st}(X)_{x,C_i}/G )
\end{align*}
show that $M_{\st}(X/G)_{\bar{x}} - M_{\st}(X)_{x}/G$ and 
$
M_{\st}(X/G)_{\bar{x},\bar{C}_i} - M_{\st}(X)_{x,C_i}/G$
have the same leading term, which has a positive coefficient. Thus, $M_{\st}(X/G)_{\bar{x}} > M_{\st}(X)_{x}/G$; as desired.

\bibliographystyle{skalpha}
\bibliography{MainBib}

\end{document}